\documentclass[12pt]{amsart}
\usepackage[margin=1in,letterpaper,portrait]{geometry}
\usepackage{hyperref}
\usepackage{graphicx}
\usepackage{color}
\usepackage{latexsym,amsmath,amssymb,verbatim}
\usepackage{mathtools}

\theoremstyle{plain}
   
   \newtheorem{theorem}{Theorem}[section]
   \newtheorem{proposition}[theorem]{Proposition}
   \newtheorem{prop}[theorem]{Proposition}
   \newtheorem{lemma}[theorem]{Lemma}
   \newtheorem{corollary}[theorem]{Corollary}
   \newtheorem{conjecture}[theorem]{Conjecture}
   
\theoremstyle{definition}
   \newtheorem{definition}[theorem]{Definition}
   \newtheorem{example}[theorem]{Example}

   \newtheorem{remark}[theorem]{Remark}

\numberwithin{equation}{section}

\newcommand{\sfrac}[3][1.2]{\scalebox{#1}{$\frac{#2}{#3}$}}
\newcommand{\bbinom}[2]{{\tbinom{#1}{#2}}} 
\newcommand{\sbinom}[2]{{\tbinom{#1}{#2}}} 
\DeclareMathOperator*{\Coef}{Coef}
\DeclareMathOperator*{\VDM}{VM}

\newcommand{\ot}{\otimes}
\newcommand{\BBB}{{\mathcal{B}}}
\newcommand{\CC}{{\mathbb {C}}}

\newcommand{\ZZ}{{\mathbb {Z}}}

\newcommand{\Jac}{{\operatorname{Jac}}}
\newcommand{\Hilb}{{\operatorname{Hilb}}}
\newcommand{\reg}{{\operatorname{reg}}}
\newcommand{\Sym}{{\operatorname{Sym}}}
\newcommand{\Frac}{{\operatorname{Frac}}}
\newcommand{\Hom}{{\operatorname{Hom}}}
\newcommand{\Tr}{{\operatorname{Tr}}}
\newcommand{\rank}{{\operatorname{rank}}}
\newcommand{\Saitodegree}{{\operatorname{\Delta}}}
\newcommand{\del}{\partial}
\newcommand{\one}{\mathbf{1}}

\newcommand{\Hderiv}{{D_H}}

\newcommand{\SSS}{{S}}

\newcommand{\lieg}{\mathfrak{g}}
\newcommand{\lieh}{V}

\newcommand{\sss}{\mathbf{s}}
\newcommand{\kk}{\mathbf{k}}
\newcommand{\mmm}{\mathbf{m}}
\newcommand{\nnn}{\mathbf{n}}

\begin{document}

\title[Invariant derivations and differential forms]
{Invariant derivations and differential forms\\ for reflection groups}

\dedicatory{To Peter Orlik and Hiroaki Terao on their 80th and 65th birthdays, respectively.}

\author{Victor\ Reiner}
\author{Anne V.\ Shepler}
\thanks{Research supported in part by NSF grants DMS-1601961,DMS-1101177,   
and Simons Foundation Grant \#429539.}
\email{reiner@math.umn.edu, ashepler@unt.edu}


\begin{abstract}
Classical invariant theory of a complex reflection group $W$ 
highlights three beautiful structures:
\begin{itemize}
\item[ ] \hspace{-3ex}$\bullet$ 
the $W$-invariant polynomials constitute a polynomial algebra, over which
\item[ ] \hspace{-3ex}$\bullet$ 
the $W$-invariant differential forms with polynomial coefficients
constitute an exterior algebra, and 
\item[ ] \hspace{-3ex}$\bullet$ 
the relative invariants of any $W$-representation
constitute a free module.
\end{itemize}
When $W$ is a {\it duality} (or {\it well-generated}) group,
we give an explicit description 
of the isotypic component within the differential forms 
of the irreducible reflection representation.
This resolves a conjecture of Armstrong, Rhoades
 and the first author, 
and relates to Lie-theoretic conjectures and
results of 
Bazlov, Broer, Joseph, Reeder, and Stembridge,
and also Deconcini, Papi, and Procesi.
We establish 
this result by examining the space of 
$W$-invariant differential derivations; these are derivations whose coefficients
are not just polynomials, 
but differential forms with polynomial coefficients.

For every complex reflection group $W$,
we show that the space of invariant differential derivations
is finitely generated as a module over
the invariant differential forms by
the basic derivations together
with their exterior derivatives.
When $W$ is a duality group,
we show that the space of invariant differential
derivations is free as a module over the exterior subalgebra
of $W$-invariant forms generated by all 
but the top-degree exterior generator.
(The basic invariant of highest degree is omitted.)

Our arguments for duality groups are case-free, i.e., 
they do not rely on any reflection group classification.
\end{abstract}

\maketitle

\vspace{-2ex}

\section{Introduction}
A celebrated result of Solomon~\cite{Solomon}
exhibits the set of differential forms invariant
under the action of a complex reflection group
as an {\em exterior algebra}.  A similar result holds when
we consider derivations instead of differential forms, 
i.e., elements of
$S(V^*)\ot V$ instead of $S(V^*)\ot \wedge V^*$, for a reflection representation $V$
with symmetric algebra $S(V^*)$.
The polynomial degrees of generators of these sets of invariants
are positioned into various combinatorial identities
expressing the geometry, topology, and representation theory
of reflection groups.
Recently, a theory of
Catalan combinatorics for reflection groups (e.g., see~\cite{ArmstrongRRhoades}) 
has prompted questions about the structure of invariant forms for other representations of
a reflection group.  Of particular interest are the {\em differential derivations},
elements of $S(V^*)\ot \wedge V^*\ot V$.
The first author together with Armstrong and Rhoades
conjectured a 
formula~\cite[Conj.~11.5${}^\prime$]{ArmstrongRRhoades} 
for the Poincar\'e polynomial
of the invariant differential derivations when the
reflection group $W$ is real, i.e., a finite Coxeter group.  
We verify this conjecture and show that
the set of invariant differential derivations, 
$$(S(V^*)\ot \wedge V^*\ot V)^W\, ,$$
is a free module over an exterior algebra
constructed from exterior derivatives $df_i$ of 
all but one of the basic invariants $f_i$;
the last basic invariant of highest polynomial degree is
omitted.
In fact, we give the 
explicit structure of the invariant differential derivations
for all complex reflection groups $W$ that are duality groups.
We also give a basis for nonduality groups.
We explain these two main results next.

\subsection*{Invariant theory of reflection groups}
Recall that a {\it reflection} on a finite dimensional
vector space $V=\CC^\ell$ is a nonidentity general linear transformation
that fixes a hyperplane in $V$ pointwise. A {\it complex reflection group} $W$
is a subgroup of $\text{GL}(V)$
generated by {\it reflections}; we assume all reflection groups
are finite.
We fix a $\CC$-basis
$x_1,\ldots,x_\ell$ of $V^*$ with dual $\CC$-basis
$y_1, \ldots, y_\ell$ of $V$ 
and identify the symmetric algebra $\SSS:=\Sym(V^*)$ 
with the polynomial ring
$\CC[x_1,\ldots,x_\ell]$,
which carries a $W$-action by linear substitutions.
Shephard and Todd~\cite{ShephardTodd} and Chevalley~\cite{Chevalley}
showed that the $W$-invariant subalgebra $\SSS^W$ is again {\it polynomial}:
$$
\SSS^W=\CC[f_1,\ldots,f_\ell]
$$ 
for certain algebraically independent polynomials
$f_1, \ldots, f_\ell$ called {\em basic invariants}.  One can
choose $f_1,\ldots,f_\ell$ homogeneous;
we assume 
$\deg(f_1) \leq \cdots \leq \deg(f_\ell)$ after re-indexing.
It follows that 
for any $W$-representation $U$,
the space of {\it relative invariants} $(\SSS \otimes U)^W$
forms a free $\SSS^W$-module of rank $\dim_\CC(U)$
(see, e.g., Hochster and Eagon~\cite[Prop.~16]{HochsterEagon}).
Note that 
we take all tensor products and exterior algebras over $\CC$
unless otherwise indicated.  We also assume all representations
are complex and finite dimensional.

\subsection*{Differential forms and derivations}
Two particular cases have received much attention.
First, when 
$U=\wedge V^*$, 
one may identify $\SSS\ot U=\SSS \otimes \wedge V^*$ 
with the $\SSS$-module of {\it differential forms with polynomial coefficients}
generated by $dx_1, \ldots, dx_\ell$.
Solomon's theorem~\cite{ShephardTodd, Solomon} 
asserts that $(\SSS \otimes \wedge V^*)^W$ is not
just a free $\SSS^W$-module, but, in fact, an exterior algebra
over $\SSS^W$ on exterior generators $df_1,\ldots,df_\ell$,
where $df_j:=\sum_{i=1}^\ell \frac{\partial f_j}{\partial x_i} \otimes x_i$
has degree $e_i:=\deg(f_i)-1$:
$$
(\SSS \otimes \wedge V^*)^W = \bigwedge_{S^W}\{ df_1, \ldots, df_{\ell}\}
\, .
$$
Second, when $U=V$, one
may identify $\SSS\ot U=\SSS \otimes V$ 
with the set of {\it derivations} $S \rightarrow S$ on $V$
generated by the partial derivatives 
$\del/\del x_1, \ldots, \del/\del x_n$.
Here, a derivation 
$\theta=\sum_{i=1}^\ell \theta^{(i)} \otimes y_i$
maps $f$ in $S$ 
to 
$\theta(f)=\sum_{i=1}^\ell 
\theta^{(i)} \sfrac{\partial}{\partial x_i}(f)$.
One may choose a homogeneous
basis $\theta_1,\ldots,\theta_\ell$
for the free $\SSS^W$-module $(\SSS \otimes V)^W$, i.e., 
\begin{equation}
\label{basic-derivations-definition}
(\SSS \otimes V)^W = S^W\theta_1 \oplus \ldots\oplus S^W \theta_\ell\,
\end{equation}
for some {\em basic derivations}
$\theta_j=\sum_{i=1}^\ell \theta_i^{(j)} \otimes y_i$
with each $\theta_i^{(j)}$ in $\SSS$  homogeneous of
fixed degree, say $e_j^*$.

\subsection*{Duality groups}
Our main results combine these contexts, with special results
for duality groups.
An irreducible complex reflection group $W$ is a {\it duality group}
if its {\it coexponents} $e_1^* \geq \cdots \geq e_\ell^*$
and {\it exponents} $e_1 \leq \cdots \leq e_\ell$ determine each other 
via the relation
$
e_i + e_i^* = h,
$
where $h:= \deg(f_\ell)=e_\ell+1$ is the largest degree of a basic invariant, 
called the {\it Coxeter number} for the duality group $W$.

Duality groups include all irreducible 
{\it real} reflection groups (i.e., finite Coxeter groups) 
as well as symmetry groups
of {\it regular complex polytopes}~\cite{Coxeter}. 
It was observed by Orlik and Solomon~\cite{OrlikSolomon} 
in a case-by-case fashion 
(using the 
classification~\cite{ShephardTodd})
that an irreducible 
complex reflection group $W$ is a duality group if and only if
it is {\it well-generated}, that is, generated by $\ell=\dim(V)$ reflections,
but we will not need this fact in the sequel.

\subsection*{Main theorems}
We consider the set $M$ of mixed forms, called {\em differential derivations},
$$
M:=\SSS \otimes \wedge V^* \otimes V.
$$
We view $M$
as an $(\SSS \otimes \wedge V^*)$-module
via multiplication in the first two tensor positions.
Its $W$-invariant subspace $M^W$
is then a $(\SSS \otimes \wedge V^*)^W$-module.
In general, $M^W$ will not
be a {\it free} module.  Nevertheless, 
our first main result asserts that for duality groups, 
$M^W$  {\it is} a free $R$-module,
where $R$ is the subalgebra of 
the invariant forms generated by all $df_i$ but the last,
$$
R:=\bigwedge_{\SSS^W}
\{df_1,\ldots,df_{\ell-1}\}\, ,
$$
with only $df_{\ell}$ omitted.  To give an $R$-basis, 
we extend the usual exterior derivative operator $d$ on $\SSS$ to
a function on $\SSS\ot \wedge V^* \ot V$ defining
$
d(f\ot \omega \ot y) = \sum_{1\leq i\leq \ell} 
\sfrac{\del f}{\del x_i} \ot (x_i\wedge \omega) \ \ot y.
$
We identify $S\ot V$ with the subspace $S\ot 1 \ot V$ of differential derivations,
$$
\begin{array}{rcl}
S \ot V &\hookrightarrow &S \ot \wedge V^* \ot V\\
f \otimes y\ &\longmapsto & \! f \, \otimes\ \ \, 1\ \  \otimes \, y ,
\end{array}
$$
and apply $d$ to derivations using this inclusion.
Our first main result is shown with case-free arguments
(it does not depend on any classification).
\begin{theorem}
\label{main-result}
For $W$ a duality (well-generated) complex reflection group,
$(\SSS \otimes \wedge V^* \otimes V)^W$ forms a
free $R$-module on the $2\ell$ basis elements
$
\{\theta_1,\ldots,\theta_\ell, \,\, d\theta_1, \ldots, d \theta_\ell\}\, .
$ 
\end{theorem}

For arbitrary complex reflection groups, we have no uniform statement like
Theorem~\ref{main-result}.  However,
we give an explicit free
$S^W$-basis for $(\SSS \otimes \wedge V^* \otimes V)^W$
in the remaining (non-duality) cases, showing this:

\begin{theorem}
\label{generation-result}
For $W$ any complex reflection group,
$(\SSS \otimes \wedge V^* \otimes V)^W$ is generated as a
module over the exterior algebra 
$(\SSS \otimes \wedge V^*)^W=\bigwedge_{\SSS^W}\{df_1,\ldots,df_{\ell}\}$ by the $2\ell$ generators
$
\{\theta_1,\ldots,\theta_\ell, \,\, d\theta_1, \ldots, d \theta_\ell\}.
$ 
\end{theorem}

To be clear:
$(\SSS \otimes \wedge V^* \otimes V)^W$ is 
not {\it freely} generated as a 
module over $\bigwedge_{\SSS^W}\{df_1,\ldots,df_{\ell}\}$
by $\{ \theta_i, d\theta_i \}_{i=1}^\ell$.

\medskip

\begin{example}
\label{rank-one-example}
A rank $\ell=1$ reflection group
$W \subset \text{GL}(V)
=\text{GL}_1(\CC)=\CC^\times$ is a {\it cyclic group}
$
W=\langle \zeta \rangle \cong \ZZ/h\ZZ
$
for some $\zeta=e^{\frac{2\pi i}{h}}$ in $\CC$.
Let $V=\CC x$ and $V^*=\CC y$ with $y$ dual to $x$.
Under the generator of $W$,
$x \mapsto \zeta^{-1}x$ and $y \mapsto \zeta y$.
Then
$$
\SSS=\CC[x],\quad
\SSS^W=\CC[f_1] \quad\text{ where }f_1=f_\ell=x^h \text{ has degree }h.
$$
The $\SSS^W$-module of invariant forms,
$(\SSS \otimes \wedge V^*)^W$, is an exterior algebra over
$\SSS^W$ generated by $df_1=h x^{h-1} \otimes x$ 
of degree $e_1=h-1$, that is,
$$
\begin{array}{rccccccl}
(\SSS \otimes \wedge V^*)^W
 &=& (\SSS \otimes \wedge^0 V^*)^W & \oplus & (\SSS \otimes \wedge^1 V^*)^W &&\\
 &=& \underbrace{\SSS^W (1 \otimes 1)}_{:=R} 
                           & \oplus & \SSS^W (x^{h-1} \otimes x) & = & \SSS^W \oplus \SSS^W \, df_1.
\end{array}
$$
On the other hand, the $\SSS^W$-module
of invariant derivations is
$$(\SSS \otimes V)^W=\SSS^W (x\otimes y) = \SSS^W \theta_1,
\quad\quad \text{for } \theta_1= x \otimes y\, \
\quad\text{ of degree } 
e_1^*=1.
$$ 
In particular, $W$ is a duality group, since $e_1^*+e_1=1+(h-1)=\deg(f_\ell)$.
Now consider the invariant differential derivations:
An easy check confirms that
$M^W=(\SSS \otimes \wedge V^* \otimes V)^W$ is a free module
over $R=\SSS^W$ with basis 
$
\{ \theta_1, d \theta_1\}.
$
Here, we have identified $\theta_1= x\otimes y$ with $x \otimes 1 \otimes y$
and $d\theta_1= 1 \otimes x \otimes y$.
Thus
$$
(\SSS \otimes \wedge V^* \otimes V)^W 
= \SSS^W(x \otimes 1 \otimes y) \oplus \SSS^W( 1 \otimes x \otimes y)\, .$$
\end{example}

\medskip

\subsection*{Outline}
Section~\ref{further-context-section} provides further context and implications of
Theorem~\ref{main-result} while Section~\ref{Lie-theory-section}
gives its relation to some theorems and conjectures in Lie theory.
We collect some tools
for establishing helpful reflection group numerology
(like Molien's Theorem and the Gutkin-Opdam Lemma)
in Section~\ref{Molien-strategy-section}
and reap that numerology in Section~\ref{Numerology-section}.
Section~\ref{Saito-criterion-section} is a slight
digression deriving basis conditions reminiscent of Saito's Criterion for
free arrangements.
A linear independence condition is given in
Section~\ref{independence-section}. 
We complete the proof of Theorem~\ref{main-result} 
in Section~\ref{well-generated-section}
after showing that duality groups exhibit auspicious numerology
in Section~\ref{numerology-works-section}.

The remainder is about nonduality groups and 
Theorem~\ref{generation-result}.
Section~\ref{2-dimensional-groups-section}
addresses rank two reflection groups, while
Section~\ref{G31-section} considers the 
Shephard and Todd group $G_{31}$ --- the only irreducible non-duality
group that is neither of rank two, nor within the Shephard and Todd infinite
family of monomial groups $G(r,p,n)$.
Section~\ref{poorly-generated-section} assembles the
proof of Theorem~\ref{generation-result},
relegating the case of $G(r,p,n)$ to
Appendix~\ref{appendix-section}.  When $p=1$ or $p=r$,
these are duality groups and covered by
Theorem~\ref{main-result}; when $1 < p <r$,
they are nonduality groups and we give
an alternate basis for the
invariant differential derivations in this appendix.

We consider some further questions in Section~\ref{questions-section}.
\section{Implications of Theorem~\ref{main-result}}
\label{further-context-section}

To provide context for Theorem~\ref{main-result},
we first note a few consequences and
special cases.

\subsection*{The case of exterior degree zero}

Theorem~\ref{main-result}
implies something that we already knew about 
$$
(\SSS \ot \wedge^0 V^* \ot V)^W=(\SSS \ot V)^W,
$$
namely, that it is a free
$\SSS^W$-module on the basis $\{\theta_j \}_{j \in \{1,\ldots,\ell\}}$;
this is true even when $W$ is not a duality group.
We will end up using this fact in the proof of the theorem.

\subsection*{The case of top exterior degree}
At the opposite extreme, Theorem~\ref{main-result} asserts
that 
$$(\SSS \ot \wedge^\ell V^* \ot V)^W
\cong (S\ot \det \ot V)^W$$ 
is free
as an $\SSS^W$-module on the 
basis $\{df_1 \cdots df_{\ell-1}\, d\theta_k \}_{k \in \{1,\ldots, \ell\}}$.
This agrees 
with the polynomial degrees of a basis
found
in~\cite{Shepler} for any reflection group $W$,
as we may view $(S\ot \det \ot V)^W$ as
the space of invariant derivations
for the ``twisted reflection representation''
$\det \ot V$, where $\det: W\mapsto \CC^*$ is the determinant
character of $W$ acting on $V$.
Indeed, if $J=\det\left( \frac{\partial f_i}{\partial x_j}\right)$ is the Jacobian
determinant of the basic invariants $f_1,\ldots,f_\ell$ 
(see Section~\ref{J-and-Q-subsection}), then 
the forms $df_1\cdots df_{\ell-1}\, d\theta_k$ have degrees
$$e_1+\ldots+ e_{\ell-1}+ e_k^*-1
=(e_1+\ldots+ e_{\ell})- (e_{\ell}+1) + e_k^*
=\deg J - (h - e_k^*)
= \deg J - e_{k}
$$
when $W$ is a duality group 
(see~\cite[Cor.~13(e)]{Shepler}).

\subsection*{The Hilbert series consequence}
Theorem~\ref{main-result} has implications for 
{\it Hilbert series} 
analogous to those given by the 
Shephard-Todd-Chevalley and Solomon theorems
with $S=\oplus_{i \geq 0}\, S_i$ graded by polynomial degree.
Just as these classical results immediately imply that
$$
\begin{aligned}
\hspace{14ex}
\label{hilb-of-polynomial-invariants}
\Hilb(\SSS^W; q)
 &:=& &\! \! \! \sum_{i \geq 0} \quad\   q^i \dim \SSS^W_i 
 &\! \! =&  \ \prod_{i=1}^\ell (1-q^{e_i+1})^{-1},& \\ 
\label{hilb-of-derivations}
\Hilb\big( (\SSS \otimes V)^W; q \big) 
 &:=& &\! \! \! \sum_{i \geq 0} \quad\  q^i \dim \left( S_i \otimes V \right)^W 
 &\! \! =&  \ \Hilb(\SSS^W; q) \cdot \sum_{i=1}^\ell\ q^{e_i^*} \ ,&
\hfill\text{and}
\\
\label{hilb-of-exterior-invariants}
\Hilb \big( (\SSS \otimes \wedge V^*)^W; q,t \big) 
 &:=& &\! \! \! \sum_{i,j \geq 0} q^i\, t^j 
 \dim ( S_i \otimes \wedge^j V^* )^W 
 &\! \! =& \ \Hilb(\SSS^W; q)  \cdot \prod_{i=1}^\ell \left( 1+q^{e_i}t \right) ,&
\end{aligned}
$$
Theorem~\ref{main-result} analogously
immediately implies that
\begin{equation}
\label{Hilbert-series-consequence}
\begin{array}{lll}
\Hilb\left( \left(\SSS \otimes \wedge V^* \otimes V \right)^W; q,t\right)
&:=\ \ \displaystyle\sum_{i,j \geq 0} \, q^i \, t^j\, 
    \dim \left( S_i \otimes \wedge^j V^* \otimes V\right)^W \\
&\ =\ \ \Hilb(S^W; q,t) \cdot  \displaystyle\sum_{i=1}^\ell\ (q^{e_i^*}+q^{e_i^*-1}t) \\
&\ =\ \ \Hilb(S^W;q) \cdot (q+t)\cdot 
       \left( \displaystyle\prod_{i=1}^{\ell-1}\ 1+q^{e_i}t \right)
       \left( \displaystyle\sum_{i=1}^\ell\ q^{e_i^*-1} \right) \, .
\end{array}
\end{equation}
Our original motivation, in fact, was
the special case of~\eqref{Hilbert-series-consequence}
for real reflection groups $W$, which appeared 
as~\cite[Conj.~11.5${}^\prime$]{ArmstrongRRhoades},
based on Coxeter-Catalan combinatorics and computer experimentation.

\section{The Lie theory connection}
\label{Lie-theory-section}
We now explain how the case of Theorem~\ref{main-result} when $W$ is a {\it Weyl group},
i.e, a finite crystallographic real reflection group,
relates to Lie-theoretic results and 
work of Bazlov, Broer, 
Joseph, Reeder, and Stembridge, and also DeConcini, Papi, and Procesi. 
Let $G$ be a simply-connected, compact simple Lie group 
with a choice of maximal torus $T$.  Denote by $\lieg$ and $\lieh$ the
complexification of their corresponding Lie algebras,
and let $W:=N_G(T)/T$ be the associated Weyl group
acting on a real vector space $V$.  
Then $G$ acts on $\wedge\lieg^*$, while $W$ acts on $\SSS:=S(\lieh^*)$ and on its
{\it coinvariant algebra}
$$
\SSS/\SSS^W_+ \cong H^*(G/T)
$$
where the last isomorphism to cohomology, due to Borel, is grade-doubling,
and where $S^W_+$ is the ideal generated by invariant polynomials of positive
degree.
Classical results (see~\cite{Reeder1}) give isomorphisms
\begin{equation}
\label{Borel-Cartan-Leray-result}
\left(\wedge\lieg^*\right)^G 
\quad \cong \quad H^*(G) 
\quad \cong \quad \left( H^*(G/T) \otimes H^*(T) \right)^W
\quad \cong \quad 
\left( S/\SSS^W_+ \otimes \wedge\lieh^*\right)^W
\end{equation}
exhibiting both of these rings as (isomorphic) exterior algebras,
with exterior generators $P_1,P_2,\ldots,P_\ell$ where $P_i$ lies in
$\left(\wedge^{2e_i+1}\lieg^*\right)^G$.  The isomorphism
is again homogeneous after doubling the grading in $S/S_+^W$.

Reeder~\cite{Reeder2} conjectured a similar relation between $G$-invariants and $W$-invariants,
relating two Hilbert series associated with a finite-dimensional $G$-representation $M$:
$$
\begin{aligned}
P_G(M;t)
  &:= \sum_{j \geq 0} \, t^j \,
       \dim (\wedge^j \lieg^* \otimes M)^G \, ,
\\
P_W(M^T;q,t)
  &:= \sum_{i,j \geq 0} \,q^i\, t^j\, 
       \dim \left( 
                   (S/S^W_+)_i
                    \otimes \wedge^j\lieh^* \otimes M^T 
                \right)^W \, .
\end{aligned}
$$
\begin{conjecture}\cite[Conj.~7.1]{Reeder2}
\label{Reeder-conjecture}
If $M$ is {\bf small}, meaning its weight space 
$M_{2\alpha}=0$ for all roots $\alpha$, one has
$$
P_G(M;t)=P_W(M^T;q,t) |_{q=t^2}.
$$
\end{conjecture}
\noindent
Various special cases of Conjecture~\ref{Reeder-conjecture} were known at the time that it was formulated.  For example, when $M$ is the trivial $G$-representation it follows
from \eqref{Borel-Cartan-Leray-result} above. 
Reeder~\cite[Cor.~4.2]{Reeder2} proved the $t=1$
specialization of Conjecture~\ref{Reeder-conjecture} and
credited it also to Kostant:   
For $M$ small, 
\begin{equation}
\label{Kostant-Reeder-ungraded-theorem}
\dim (\wedge \lieg^* \otimes M)^G 
=P_G(M;t)|_{t=1}
=P_W(M^T;q,t) |_{q=t=1}
=\dim \left( S/S^W_+ \otimes \wedge\lieh^* \otimes M^T \right)^W.
\end{equation}
The type $A$ special case was also known to follow from the ``first-layer'' formulas of Stembridge~\cite{Stembridge}.  The type $B$ special case was 
recently\footnote{DeConcini and Papi, personal communication, 2016.} 
verified in work of DeConcini and Papi, and 
Stembridge
independently\footnote{Stembridge, personal communication, 2016.}   
verified 
Conjecture~\ref{Reeder-conjecture} case-by-case
for all types.

A further bit of motivation comes from a generalization of
Chevalley's restriction theorem due to Broer~\cite{Broer}.  Chevalley's result asserts
that restriction $\lieg^* \rightarrow \lieh^*$ induces an isomorphism of polynomial rings
\begin{equation}
\label{Chevalley-restriction-theorem}
S(\lieg^*)^G \rightarrow \SSS^W,
\end{equation}
while Broer~\cite{Broer} showed more generally that, for any small $G$-module $M$, restriction
also induces an isomorphism 
(of modules over the polynomial rings in~\eqref{Chevalley-restriction-theorem})
$$
(S(\lieg^*) \otimes M)^G \rightarrow (\SSS \otimes M^T)^W.
$$

Broer's result 
suggested to the authors the following enhanced
version of Conjecture~\ref{Reeder-conjecture}.

\begin{conjecture}
\label{extended-Reeder-conjecture}
{(Enhanced Reeder Conjecture)}
For a small $G$-representation $M$, there is an isomorphism 
$$
(\wedge \lieg^* \otimes M)^G 
\cong
\left(  S/S^W_+ \otimes \wedge\lieh^* \otimes M^T \right)^W
$$
of modules over the exterior algebra in~\eqref{Borel-Cartan-Leray-result}
which is degree-preserving after doubling the grading in $S/S^W_+$.
\end{conjecture}
While this paper was under review,
DeConcini and Papi~\cite[p.~259]{DeConciniPapi} showed that {\em not all} small $G$-representations $M$
satisfy Conjecture~\ref{extended-Reeder-conjecture}, 
and one asks, ``For which small $G$-representations does the Enhanced Reeder Conjecture~\ref{extended-Reeder-conjecture} hold?"
DeConcini and Papi~\cite[Thm.~2.2 and Cor.~6.6]{DeConciniPapi} 
proved the conjecture holds for two particularly important cases
of small $G$-representations,
namely, the {\it adjoint representation}\footnote{It is exactly in the case of the adjoint representation
that Bazlov~\cite{Bazlov} proved Conjecture~\ref{Reeder-conjecture}, and he credits this special case of
the conjecture to Joseph~\cite{Joseph}.}  $M=\lieg$ and the {\it little adjoint representation},
which are the $\lieg$-irreducibles whose highest weights are the highest root and highest short root, respectively.  The adjoint case  $M=\lieg$ connects 
our work to the following result of DeConcini, Papi, and Procesi.

\begin{theorem}{\cite[Thm~1.1]{DPP}}
\label{DPP-theorem}
Regard
$
(\wedge \lieg^* \otimes \lieg)^G
$
as a module over the exterior subalgebra $R$ 
of $(\wedge \lieg^*)^G$ generated by $P_1,P_2,\ldots,P_{\ell-1}$,
via multiplication in the first tensor factor.
Then $
(\wedge \lieg^* \otimes \lieg)^G
$
is free as an $R$-module, with 
basis elements $\{f_i,u_i\}_{i=1}^\ell$ of degrees
$\deg(f_i)=2e_i, \deg(u_i)=2e_i-1$.
\end{theorem}

\noindent
An alternate proof of Theorem~\ref{DPP-theorem} follows from
our Theorem~\ref{main-result}, using 
the adjoint special case of Conjecture~\ref{extended-Reeder-conjecture}, 
that is, \cite[Thm. 2.2]{DeConciniPapi},
after modding out by $S^W_+$
and bearing in mind that 
$\{e^*_i\}_{i=1}^\ell=\{e_i\}_{i=1}^\ell$ for Weyl groups $W$.

\section{Degree sums and the Gutkin-Opdam Lemma}
\label{Molien-strategy-section}

Before determining bases for the invariant differential derivations, 
we recall some tools for investigating the relevant numerology, most notably a 
useful lemma for finding the sum of degrees in a basis.

\subsection*{Degree sum}
Let $k$ be a field, and let $A$ be a graded $k$-algebra and integral domain. 
Consider a free graded $A$-module 
$M \cong A^p$ of finite rank, say
with homogeneous basis $m_1,\ldots, m_p$.
The (unordered) list of degrees 
$\deg(m_1),\ldots,\deg(m_p)$ 
are uniquely determined by
the quotient of Hilbert series
$$
\sum_{i=1}^p q^{\deg(m_i)}=
\Hilb(M,q)/\Hilb(A,q).
$$
Thus we may assign to any such $M$ the {\em degree sum}
$$
\Delta_A(M):=\sum_{i=1}^p \deg(m_i) = 
\left[ \sfrac{\partial}{\partial q} 
\frac{\Hilb(M,q)}{\Hilb(A,q)} \right]_{q=1}.
$$
If one knows this degree sum {\em a priori},
then one may determine an explicit $A$-basis for $M$
by just checking independence over the 
fraction field:

\begin{lemma}
\label{commutative-algebra-lemma}
Let $A$ be a graded $k$-algebra and integral domain
and $M \cong A^p$ a free graded $A$-module.
A set of homogeneous
elements $\{n_1, \ldots, n_p\}$ in $M$ with 
$\sum_{i=1}^p \deg(n_i)=\Delta_A(M)$ 
is an $A$-basis
for $M$ if and only if it is linearly independent over the fraction field $K=\Frac(A)$.
\end{lemma}
\begin{proof}
The forward implication is clear.  For the reverse implication, note that
linear independence is equivalent to
nonsingularity of the matrix $B$ in $A^{p \times p}$ with $\nnn = B \mmm$,
for  $\mathbf{n}=[n_1,\ldots,n_p]^T$ and $\mathbf{m}=[m_1,\ldots,m_p]^T$.
Since $\det(B) \neq 0$, its expansion contains a nonzero term 
indexed by a permutation $\pi$ 
with $b_{ i , \pi(i) } \neq 0$ for each $i=1,2,...,p$;
after re-indexing, 
one may assume $\pi$ is the identity permutation.
Hence $\deg(n_i) =\deg( b_{ i , i} ) + \deg(m_i)$ 
for $i=1,\dots,p$.
Since $\sum_i \deg(n_i)=\Delta_A(M)=\sum_i \deg(m_i)$,
one has that $\deg(n_i)=\deg(m_i)$ for each $i$.   Thus after re-ordering the
rows and columns of $B$ in increasing order of degree, $B$ will be block
upper triangular, with each diagonal block an invertible matrix 
with entries in $k$.
Therefore $B$ gives an $A$-module automorphism of $M$  sending the
$A$-basis $\mmm$ to $\nnn$.
\end{proof}

\subsection*{Modules over the Invariant Ring}
As mentioned in the Introduction, 
a result of Hochster and Eagon
implies that for any 
representation $U$ of a complex reflection group $W$,
the set of all {\it relative invariants} $M=(\SSS \ot U)^W$ is 
a free module of finite rank $p=\dim_\CC U$
over the graded $k$-algebra $A=S^W$.
We introduce an abbreviation for the above degree sum:
\begin{definition}
\label{Saito-degree-def}
Let $W$ be a complex reflection group.
For any $W$-representation $U$, set
$$
\Delta(U):=\Delta_{S^W}\big( (S\ot U)^W\big)\, 
=\sum_{1\leq i \leq p}\deg\psi_i 
$$
for any 
$\SSS^W$-basis $\{\psi_i\}_{i=1}^p$ of $(\SSS\ot U)^W$.
\end{definition}

\subsection*{Local Data}
We next review an {\it a priori} calculation for the degree sum $\Delta(U)$, 
Lemma~\ref{Opdam-lemma} below, due originally to Gutkin~\cite{Gutkin},
and later rediscovered by Opdam~\cite[Lemma~2.1]{Opdam};
see also Brou\'e~\cite[Prop. 4.3.3 and eqn. (4.6)]{Broue}.

The formula for $\Delta(U)$ is expressed in what is sometimes called the {\it local data}
for $U$ at each reflecting hyperplane $H$ of $W$.
The pointwise stabilizer subgroup $W_H$ in $W$ of $H$ is cyclic,
say of order $e_H$; 
note that $e_H$ is the maximal order of a reflection in $W$ 
fixing $H$ pointwise.  The $W_H$-irreducible representations 
are the powers $\{\det^j\}_{j=0}^{e_H-1}$
of the $1$-dimensional (linear) 
character $\det:= \det \downarrow_{W_H}^{W}$ 
restricted from $W$ to $W_H$ acting on $V$.  
It is convenient to introduce
the {\it representation ring} 
$$R(W_H):=\ZZ[v]/(v^{e_H}-1),$$ where $v^j$ represents
the class of the
$1$-dimensional representation $\det^j$,
and define a $\ZZ$-linear
functional  
$$D_H: R(W_H) \rightarrow \ZZ,
\quad v^j  \mapsto  j\, .$$
Then for any $W$-representation $U$, 
the functional $\Hderiv$ on the restricted representation
$U\downarrow_{W_H}^{W}$
can be expressed in terms of the inner products
$
\mu_{H,j}:= \langle U\downarrow_{W_H}^{W}, {\det}^j \rangle_{W_H}$ as
$$
\Hderiv \left(U\downarrow_{W_H}^{W} \right)=\sum_{j=0}^{e_H-1} j \cdot \mu_{H,j}\, .
$$

\begin{lemma}(Gutkin-Opdam Lemma)
\label{Opdam-lemma} 
Let $U$ be a representation
of a complex reflection group $W$.
Then
$$
\Delta(U) = \sum_H \Hderiv \left(U\downarrow_{W_H}^{W}\right),
$$
where the sum runs over all reflecting hyperplanes $H$ for $W$.
\end{lemma}

Lemma~\ref{Opdam-lemma} can be deduced (see Brou\'e \cite[\S 4.5.2]{Broue}) from the following
standard variant of Molien's theorem.  In its statement, $\Tr$ indicates trace.
\begin{lemma} \cite[Lem. 3.28]{Broue}
For any $W$-representation $U$,
\begin{equation}
\label{standard-Molien-variant}
\Hilb\left((S \ot U)^W,q \right)
=\tfrac{1}{|W|\rule{0ex}{1.5ex}} \sum_{w \in W} \frac{\Tr_U(w^{-1})}{\det(1-qw)}\ .
\end{equation}
\end{lemma}
For a complex reflection group $W$, taking $U$ to be the trivial representation in Lemma~\ref{standard-Molien-variant}
immediately implies the well-known fact that
$$
\prod_{i=1}^\ell \frac{1}{1-q^{e_i+1}}
=\Hilb(S^W,q)
=\tfrac{1}{|W|} \sum_{w \in W} \frac{1}{\det(1-qw)}\ 
$$
as $S^W=\CC[f_1,\ldots,f_\ell]$ with
$\deg(f_i)=e_i+1$.
As noted by Shephard and Todd \cite[\S8]{ShephardTodd}, 
comparing the first two coefficients in the
Laurent expansions about $q=1$ on the left and right immediately
gives these facts:
\begin{eqnarray}
\label{Shephard-Todd-product-fact}
|W|&=&\prod_{i=1}^\ell (e_i+1),\\
\label{Shephard-Todd-sum-of-exponents-fact}
N:= \#\{\text{reflections in }W\} &=&
\sum_{i=1}^\ell e_i\, .
\end{eqnarray}

\section{Numerology from the Gutkin-Opdam lemma}
\label{Numerology-section}
Again, $W$ is a complex reflection group.
This section harvests numerology from Lemma~\ref{Opdam-lemma}.

\subsection*{Number of Reflecting Hyperplanes}
We first see  that 
Lemma~\ref{Opdam-lemma} implies that the coexponents $e^*_i$
sum to the number $N^*$ of reflecting hyperplanes for $W$.

\begin{example}
Let 
$U=V$, the reflection representation.  Each reflecting hyperplane $H$ has
$\mu_{H,j}(V)=0$ for $j \neq 0,1$, with $\mu_{H,0}=\ell-1$ and $\mu_{H,1}=1$,
and hence $\Hderiv(V\downarrow_{W_H}^{W})=1$.
Thus in this case, Lemma~\ref{Opdam-lemma} implies that the coexponents $e^*_i:=\deg(\theta_i)$ for the 
$S^W$-basis $\{\theta_i\}_{i=1}^\ell$ of $(S \ot V)^W$ satisfy the well-known formula (e.g., see~\cite[p.~130]{BMR})
\begin{equation}
\label{sum-of-coexponents-is-N*}
N^*:=
\#\{\text{reflecting hyperplanes in }W\}
= \sum_H 1 
=\Delta(V) 
=
\sum_{i=1}^\ell e^*_i\, .
\end{equation}
\end{example}

\subsection*{Graded representations}
In order to apply Lemma~\ref{Opdam-lemma} to {\it graded} 
$W$-representations, 
we consider the {\it graded} representation ring
$R(W_H)[[t]]$,
that is, the ring of graded virtual $W_H$-characters. 
We also extend $\Hderiv$ to a map
$R(W_H)[[t]] \rightarrow \ZZ$ coefficientwise, defining
$\Hderiv \left( \sum_{k \geq 0} t^k \chi \right) := \sum_{k \geq 0} t^k \Hderiv(\chi)$.
The sum in the following corollary is over all reflecting hyperplanes $H$
of $W$.

\begin{corollary}
\label{general-exterior-tensor-degree-sum-corollary}
For any $W$-representation $U$,
$$
\sum_{m=0}^\ell \Delta( \wedge^m V^* \otimes U)\, t^m =
(1+t)^{\ell-1} \sum_H \Hderiv \left( (1+v^{e_H-1} t) \sum_{j=0}^{e_H-1} \mu_{H,j}\, v^j \right).
$$
\end{corollary}
\begin{proof}
Recall that each $U\downarrow_{W_H}^{W} = \sum_{j=0}^{e_H-1} \mu_{H,j}\, v^j$ in $R(W_H)$.
The restriction $V^*\downarrow^W_{W_H}$  is a sum of $\ell-1$ copies
of the trivial representation and one copy of $\det^{e_H-1}$.  
Hence 
the restriction 
$\wedge V^* \downarrow^W_{W_H}$ will be represented by
$(1+t)^{\ell-1} (1+v^{e_H-1}t)$, and 
$(\wedge V^* \otimes U) \downarrow^W_{W_H}$ will be represented by
$
(1+t)^{\ell-1} (1+v^{e_H-1}t) 
\sum_{j=0}^{e_H-1} \mu_{H,j}\, v^j
$
in $R(W_H)[[t]]$.
The result then follows from Lemma~\ref{Opdam-lemma}.
\end{proof}

\subsection*{Solomon's theorem}
We illustrate in the next example how
Corollary~\ref{general-exterior-tensor-degree-sum-corollary} 
gives
Solomon's result~\cite{Solomon}
that the space of $W$-invariant differential forms 
is generated by $df_1,\ldots, df_\ell$
as an exterior algebra\footnote{For an alternate geometric proof sketch of this result, see Berest, Etingof, and Ginzburg~\cite[Remark~1.17]{BEG}.}.
\begin{example}
We show that $(S \otimes \wedge^m V^*)^W$ has $S^W$-basis 
$\{ df_I \}_{I \in \binom{[\ell]}{m}}$
where $\binom{[\ell]}{m}$ denotes the collection of all $m$-element subsets
$I=\{i_1,\ldots,i_\ell\}$ of the set $[\ell]:=\{1,2,\ldots,\ell\}$,
with $1\leq i_1<\ldots<i_m\leq \ell$, and where
$$
df_I:=
df_{i_1} \wedge\cdots \wedge df_{i_m}.
$$
Apply Corollary~\ref{general-exterior-tensor-degree-sum-corollary}  
to the trivial representation $U$,
obtaining
$$
\begin{aligned}
\sum_{m=0}^\ell \Delta( \wedge^m V^*) t^m 
=(1+t)^{\ell-1} \sum_H \Hderiv (1+v^{e_H-1} t) 
=(1+t)^{\ell-1} \sum_H (e_H-1)\, t 
=Nt(1+t)^{\ell-1} 
\end{aligned}
$$
where the last equality uses Equation~\eqref{Shephard-Todd-sum-of-exponents-fact}.  
Therefore
$
 \Delta( \wedge^m V^*) = 
  \binom{\ell-1}{m-1} N.
$
Note that the sum of degrees of the elements in the alleged basis  $\{ df_I \}_{I \in \binom{[\ell]}{m}}$
matches this:
 $$
 \sum_{I \in \binom{[\ell]}{m}} \sum_{i \in I} e_i
 =\sum_{i=1}^\ell e_i\, \#\left\{I \in \sbinom{[\ell]}{m}: i \in I\right\}
 =N \sbinom{\ell-1}{m-1}
 =\Delta(\wedge^m V^*).
$$
Hence by 
Lemma~\ref{commutative-algebra-lemma}, 
it suffices only to check the linear independence
of $df_I$ over 
$K=\Frac(S)$.  As observed by Solomon, the $m$-fold wedge products 
of
the elements $df_1,\ldots,df_\ell$ 
form
a $K$-basis for  $K \ot \wedge^m V^*$ if and only their top wedge is nonvanishing:
$$
df_1\wedge \cdots\wedge  df_\ell 
= \det\left(\frac{\partial f_i}{\partial x_j}\right) \ot x_1\wedge \cdots\wedge x_\ell \neq 0.
$$
But this follows immediately
from the Jacobi Criterion~\cite{Jacobi}:
The algebraic independence of $f_1,\ldots,f_\ell$ 
implies
that the matrix of coefficients
$(\frac{\partial f_i}{\partial x_j})$
of $df_1, \ldots, df_\ell$
is nonsingular.
\end{example}

\subsection*{Numerology of differential derivations}
Consider the $W$-representation 
$U:=\wedge^m V^* \ot V$
of dimension 
$$
p:= \ell \sbinom{\ell}{m}
=\dim_\CC (\wedge^m V^* \ot V) \, .$$
We show that $\Delta(\wedge^m V^*\ot V)$
depends only on $N$, $N^*$, and $\ell:=\dim_\CC V$.
\begin{prop}
\label{Molien-derivative-prop}
For any complex reflection group $W$
and $1\leq m\leq \ell$,
$$
\Saitodegree(\wedge^m V^*\ot V)=
(\ell-1)\sbinom{\ell-1}{m-1}N + \sbinom{\ell-1}{m}N^*\, .
$$
\end{prop}

\begin{proof}
We take $U=V$ in 
Corollary~\ref{general-exterior-tensor-degree-sum-corollary}:
$$
\begin{aligned}
\sum_{m=0}^\ell \Delta( \wedge^m V^* \otimes V) t^m 
&=(1+t)^{\ell-1} \sum_H \Hderiv \left( (1+v^{e_H-1} t) ( \ell-1 + v ) \right) \\
&=(1+t)^{\ell-1} \sum_H \Hderiv \left( \ell-1 + (\ell-1)tv^{e_H-1}+ v + t \right) \\
&=(1+t)^{\ell-1} \sum_H \big( (\ell-1)t(e_H-1) + 1 \big)
= (1+t)^{\ell-1}\big( (\ell-1)Nt+N^* \big),
\end{aligned}
$$
using Equations~\eqref{Shephard-Todd-sum-of-exponents-fact} and~
\eqref{sum-of-coexponents-is-N*}.  
The proposition now follows  from the binomial theorem.
\end{proof}

Proposition~\ref{Molien-derivative-prop} and
Lemma~\ref{commutative-algebra-lemma} imply a corollary used repeatedly in the proofs of Theorems~\ref{main-result} and \ref{generation-result}.

\begin{corollary}
\label{key-numerology-corollary}
Any set of homogeneous elements $\{\psi_i\}_{i=1,2,\ldots,\ell\binom{\ell}{m}}$
in
$(S \ot \wedge^m V^* \ot V)^W$ with
$$\sum_i \deg(\psi_i)=
\Saitodegree(\wedge^m V^*\ot V)=
(\ell-1)\sbinom{\ell-1}{m-1}N + \sbinom{\ell-1}{m}N^*
$$
forms an $S^W$-basis for $(S \ot \wedge^m V^* \ot V)^W$ if and only
if they are linearly independent over $K=\Frac(S)$.
\end{corollary}

\section{Digression: A Saito Criterion}
\label{Saito-criterion-section}

This section, although not needed for the sequel, extends Corollary~\ref{key-numerology-corollary} to a condition similar to {\it Saito's criterion} for 
free hyperplane arrangements~\cite[\S4.2]{OrlikTerao}
for any complex reflection group $W$.
We first
recall some facts about the coefficient matrices
for the differential forms $df_i$ and
basic derivations $\theta_i$.

\subsection*{Jacobian matrix and product of hyperplanes}
\label{J-and-Q-subsection}

Recall 
the defining polynomial $Q$ and the Jacobian polynomial $J$
for a complex reflection group $W$:
$$
\begin{aligned}
Q&:=\prod_{H} l_H,& \quad\text{and}\quad
J&:=\prod_{H} l_H^{\, e_H-1}\, .
\end{aligned}
$$
Here, the product is taken over all reflecting hyperplanes $H=\ker l_H$ in $V$
for some choice of 
linear forms $l_H\in V^*$ with $e_H=|\text{Stab}_W(H)|$.
Note that $Q$ and $J$ are only well-defined
up to nonzero scalars.
Let $\Jac(f)$ and $M(\theta)$, respectively, be the matrices in 
$S^{\ell \times \ell}$ that express
$\{df_i\}_{i\in [\ell]}$ and $\{\theta_i\}_{i\in [\ell]}$
in the $\SSS$-bases $\{1\ot x_i\}_{i \in [\ell]}$ and $\{1\ot y_i\}_{i \in [\ell]}$
for $\SSS \otimes V^*$ and $\SSS \otimes V$, respectively.
Steinberg~\cite{Steinberg} and 
Orlik and Solomon~\cite[\S2]{OrlikSolomon}, respectively, showed that
$$
J=\det(\Jac(f))\quad\quad\text{and}\quad\quad
Q=\det(M(\theta))\, ,
$$
so that
\begin{equation}
\begin{aligned}
\label{N-and-N*-definition}
N 
  &= \sum_{i=1}^\ell {e_i}=\deg(J)
=\ \# \{\text{reflections in }W\}\, \quad\text{and}\\
N^* &=
    \sum_{i=1}^\ell {e^*_i}=\deg(Q)
=\  \#\{\text{reflecting hyperplanes for }W\}
\, .
\end{aligned}
\end{equation}
Note that Terao~\cite{Terao} also
showed that invariant derivations
$\theta_i$ also give an $\SSS$-basis for the {\it module of derivations}
of the reflection hyperplane arrangement of $W$; see 
also~\cite[\S4.1,~\S6.3]{OrlikTerao}.

When $W$ is a duality group, observe that (by definition)
\begin{equation}
\label{h-expressed-in-N-and-N*}
h:=\deg(f_\ell)=\frac{1}{\ell}\sum_{i=1}^\ell (e_i+e^*_i) = \frac{N+N^*}{\ell}
\, .
\end{equation}

\subsection*{Matrix of Coefficients}
We capture the coefficients
of any set of invariant differential derivations
in a matrix of coefficients.
Consider the
obvious free $\SSS$-basis for
$\SSS\ot \wedge^m V^* \ot V$
given by
\begin{equation}
\label{obvious-basis}
\{dx_{I} \ot y_j: 1\leq i,j\leq \ell\}
\qquad\qquad\text{with}\qquad
dx_{I}:=1\ot x_{i_1}\wedge \cdots\wedge x_{i_m}
\end{equation}
for $m$-subsets $I=\{i_1 < \cdots < i_m\}$ of $[\ell]$
and $[\ell]:=\{1,\ldots, \ell\}$.
Given a collection $\BBB\subset
(S \ot \wedge^m V^* \ot V)^W$,
let  $\Coef(\BBB)$ denote its coefficient matrix in $\SSS^{p \times p}$
with respect to the $\SSS$-basis in \eqref{obvious-basis}.

\begin{lemma}
\label{mandatory-divisibility-lemma}
For any $\BBB=\{\psi_i\}_{i=1}^p \subset 
(\SSS\ot \wedge^m V^* \ot V)^W$, 
the product
$J^{(\ell-1)\binom{\ell-1}{m-1}} Q^{\binom{\ell-1}{m}}$ divides
$\det\Coef(\BBB)$.
\end{lemma}

\begin{proof}
Fix a reflecting hyperplane $H$ in $V$ for $W$ and 
a reflection $s$ in $W$ of maximal order $e_H$ fixing $H$.
Choose coordinates
$x_1, \ldots, x_\ell$ of $V^*$ so that
$l_H=x_1$ and $s$
acts diagonally
with nonidentity eigenvalue $\xi$: 
$$
s(y_i)=\begin{cases} 
\xi^{-1} y_1 &\text{ if }i=1, \\
y_i &\text{ if }i \neq 1,
\end{cases}
\qquad \text{ and } \qquad
s(x_i)=\begin{cases} 
\xi x_1 &\text{ if }i=1, \\
x_i &\text{ if }i \neq 1\, .
\end{cases}
$$
Each row of $\Coef(\BBB)$ 
lists the coefficients $f_{I,j}$ of some invariant 
differential derivation
$\psi_i=\sum
f_{I,j}\, dx_I\otimes y_j$
in $(\SSS\ot \wedge^m V^* \ot V)$,
while each column of
$\Coef(\BBB)$ is indexed by a pair $(I,j)$
for $I$ an $m$-subset of $[\ell]$
and $j\in[\ell]$.
Note two observations:
\begin{itemize}
\item
For each pair $(I,j)$ with $j \neq 1$ but $1\in I$,
the polynomial $x_1^{e_H-1}$ divides $f_{I,j}$ since
$$
s(dx_I \otimes y_j) = \xi(dx_I\otimes y_j)
\quad\text{implies that}\quad
s(f_{I,j})=\xi^{-1} f_{I,j}\, ;
$$
thus $(\ell-1)\binom{\ell-1}{m-1}$ different columns 
of the matrix
$\Coef(\BBB)$ are divisible by $x_1^{e_H-1}$.
\item
For each pair $(I,j)$ with $j = 1$ but $1\in I$,
the polynomial $x_1$ divides
$f_{I,j}$,
since
$$
s(dx_I\otimes y_j) = \xi^{-1}(dx_I\otimes y_j)
\quad\text{implies that}\quad
s(f_{I,j})=\xi f_{I,j}\, ; 
$$
thus $\binom{\ell-1}{m}$ different columns of the matrix
$\Coef(\BBB)$ are divisible by $x_1$.
\end{itemize}
Hence $\ell_H=x_1$ when raised to the power
${(e_H-1)(\ell-1)\binom{\ell-1}{m-1}+\binom{\ell-1}{m}}$ divides
$\det\Coef(\BBB)$.  
This holds for each reflecting hyperplane $H$,
and therefore unique factorization implies that 
$J^{(\ell-1)\binom{\ell-1}{m-1}} Q^{\binom{\ell-1}{m}}$ divides $\det\Coef(\BBB)$.
\end{proof}

\subsection*{Saito-like Criterion}
Again, let $K=\CC(x_1,\ldots,x_\ell)$
be the fraction field of $S=\CC[x_1,\ldots,x_\ell]$.

\begin{corollary}
\label{Saito-criterion}
For a homogeneous subset $\BBB=\{\psi_i\}_{i=1}^p$ of
$(\SSS\ot \wedge^m V^* \ot V)^W$, 
the following are equivalent:
\begin{enumerate}
\item[(a)]
$\BBB$ forms an $\SSS^W$-basis for 
$(\SSS\ot \wedge^m V^* \ot V)^W$.
\rule{0ex}{3ex}
\item[(b)] $\det\Coef(\BBB)$ is nonzero of degree
$(\ell-1)\binom{\ell-1}{m-1}N+\binom{\ell-1}{m}N^*$.
\rule{0ex}{3ex}
\item[(c)] \rule{0ex}{3ex}
$\det\Coef(\BBB)=
c\cdot J^{(\ell-1)\binom{\ell-1}{m-1}} Q^{\binom{\ell-1}{m}}
$
for some nonzero scalar $c$ in $\CC$.
\rule{0ex}{3ex}
\rule{0ex}{3ex}
\item[(d)] \rule{0ex}{3ex}
$
\sum_{i=1}^p \deg(\psi_i) = 
(\ell-1)\binom{\ell-1}{m-1}N+\binom{\ell-1}{m}N^*$
and $\BBB$ is $K$-linearly independent in the space $K \ot \wedge^m V^* \ot V$.
\rule{0ex}{3ex}
\end{enumerate}
\end{corollary}
\begin{proof}
Corollary~\ref{key-numerology-corollary} gives the equivalence
of (a) and (d), linear algebra gives the equivalence of (d) and (b), and Lemma~\ref{mandatory-divisibility-lemma}
gives the equivalence of (b) and (c).
\end{proof}

\begin{remark}
An argument with Cramer's rule shows that 
Corollary~\ref{Saito-criterion}~(c) implies~(a)
directly
without using Corollary~\ref{key-numerology-corollary},
or appealing to any Hilbert series argument. 
Indeed,~(c) implies that
$\BBB$ {\it spans}
$(\SSS \otimes \wedge^m V^* \otimes V)^W$ over $\SSS^W$,
as we explain next.
Label the $\SSS$-basis elements 
$dx_I\otimes y_j$ of 
$\SSS \otimes \wedge^m V^* \otimes V$ as $z_1,\ldots,z_p$ 
for convenience.
Then the matrix  $\Coef(\BBB)$ in $\SSS^{p \times p}$
expresses the elements 
$\BBB=\{\psi_i\}_{i=1}^p$ in the $\SSS$-basis $\{z_i\}_{i=1}^p$:
\begin{equation}
\label{change-of-basis-matrix-equation}
\psi_j = \sum_{i=1}^p \Coef(\BBB)_{ij} \cdot z_i\, .
\end{equation}
To show that a typical element 
$\sum_{i=1}^p s_i z_i$ in $(\SSS \otimes \wedge^m V^* \otimes V)^W$
lies in the $\SSS^W$-span of $\BBB$,
find $k_i$ in the fraction field $K$ of $\SSS$ 
(using  $\det\Coef(\BBB) \neq 0$)
with
\begin{equation}
\label{fraction-field-expression}
\sum_{i=1}^p s_i z_i = \sum_{j=1}^p k_j \psi_j\, .
\end{equation}
We may assume each $k_j$ lies in $K^W$, else apply 
the symmetrizer $\frac{1}{|W|} \sum_{g \in W} g(-)$ to
\eqref{fraction-field-expression}
and use the $W$-invariance of $\sum_i s_i z_i$ and of each $\psi_j$.
To show the $k_j$ actually lie in $\SSS^W$,  
substitute~\eqref{change-of-basis-matrix-equation}
into~\eqref{fraction-field-expression}, giving a matrix
equation relating the column vectors 
$\sss=\left[s_1,\ldots,s_p\right]^t$ and 
$\kk=\left[k_1,\ldots,k_p\right]^t$ in $K^p$:
$$
\sss=\Coef(\BBB) \cdot \kk.
$$
Cramer's rule then implies that
\begin{equation}
\label{Cramer-consequence}
k_i = \frac{\ \ \det \Coef(\BBB_{(i)})} 
                      {\det\Coef(\BBB)}
\end{equation}
where the numerator matrix $\Coef(\BBB_{(i)})$ 
is obtained from $\Coef(\BBB)$ 
by replacing its $i^{th}$ column with $\sss$.  
Then since $\Coef(\BBB_{(i)})$ expresses the
elements $\psi_1,\ldots,\psi_{i-1},\sum_i s_i z_i,\psi_{i+1},\ldots,\psi_p$ 
of
$(\SSS \otimes \wedge^m V^* \otimes V)^W$ in terms of the basis $\{z_i\}$,
Lemma~\ref{mandatory-divisibility-lemma} implies that 
its determinant is divisible by the nonzero polynomial
$\det\Coef(\BBB)$.  
Thus the right side of \eqref{Cramer-consequence} lies in $\SSS$, 
so that its left side $k_i$ lies in 
$K^W \cap \SSS = \SSS^W$, as desired.  
\end{remark}

\section{Independence over the fraction field}
\label{independence-section}

In this section, we use Springer's theory of regular elements
to investigate differential derivations
with coefficients in the fraction field
$K=\CC(x_1,\ldots,x_\ell)$
of $S=\CC[x_1,\ldots,x_\ell]$.
We will later show that 
Theorem~\ref{main-result}
(for duality groups) follows from
a more general statement
established in this section for arbitrary
reflection groups, Theorem~\ref{linear-independence-theorem},
describing a $K$-vector space basis 
for $(K \ot \wedge^m V^*  \ot V)^W$.

We first give a definition and a lemma.
Recall the notation $V^{\reg}$ 
for the complement within $V$ of the union of 
all reflecting hyperplanes for $W$, that is, the subset of vectors
in $V$ having regular $W$-orbit.
Recall the notation 
$
df_I:=df_{i_1}\wedge \cdots\wedge df_{i_m}\in \SSS\ot \wedge^m V^*
$
for subsets 
$I=\{i_1 < \ldots < i_m\} \subset [\ell]:=\{1,2,\ldots,\ell\}$,
and the notation  $\binom{[\ell]}{m}$ for the collection of all 
$m$-subsets $I \subset [\ell]$.


\begin{definition}
\label{Euler-map-definitions}
Define a $K$-linear map
$$
K \otimes V^*  \overset{E}{\longrightarrow} K
\quad\quad\text{by}\quad E(dx_i)=E(1\ot x_i)=x_i\, .$$ 
One also has a $K$-linear map
$
K \otimes V^* \otimes V  \xrightarrow{\ E \otimes \one_V }
K \otimes V\, .
$
\end{definition}

Note that by {\it Euler's identity}, for any homogeneous $f$ in $\SSS$,
\begin{equation}
\label{Euler-and-d-on-polynomial}
E(df) =\deg(f) \cdot f.
\end{equation}
Likewise, for $\theta=\sum_{j=1}^\ell \theta^{(j)}\otimes y_j$ 
in $\SSS \otimes V$ with each $\theta^{(j)}$ homogeneous 
of fixed degree $\deg(\theta)$,
\begin{equation}
\label{Euler-and-d-on-derivation}
(E\ot {\bf 1}_V )(d \theta) = \deg(\theta) \cdot \theta.
\end{equation}

We first give a 
$K$-basis for differential derivations with coefficients in $K$.
Recall from~\eqref{basic-derivations-definition}
that $\theta_1,\ldots,\theta_\ell$ are any choice of
basic derivations, i.e., any choice of homogeneous 
$S^W$-basis of $(S \otimes V)^W=(S \otimes \wedge^0 V^* \otimes V)^W$
(identifying $S\otimes V$ with $S\otimes 1 \otimes V$).

\begin{lemma}
\label{more-obvious-basis-lemma}
For each $0 \leq m \leq \ell$, 
the following set gives a $K$-basis for 
$K \ot \wedge^m V^*  \ot V$:
\begin{equation}
\label{more-obvious-basis}
\widetilde{\BBB}^{(m)} 
:= \big\{ df_{I} \, \theta_k \big\}_{ 
I \in \tbinom{[\ell]}{m}, \ k \in [\ell]}\ .
\end{equation}
Furthermore, elements of $\SSS \ot V^* \ot V$
can be expressed in the $K$-basis $\widetilde{\BBB}^{(1)}$ 
with coefficients in $(JQ)^{-\ell}\rule{0ex}{1.5ex}S$.
\end{lemma}

\begin{proof}
The matrix that expresses 
$\widetilde{\BBB}^{(m)}$
in the usual $\SSS$-basis
$
\{ dx_{I}\ot y_j : I \in \tbinom{[\ell]}{m}, \ j \in [\ell] \}
$
of $\SSS \ot \wedge^m V^*  \ot V$
is the tensor product of the matrices $\wedge^m(\Jac(f)) \otimes M(\theta)$, where $\wedge^m(\Jac(f))$ is
the $m^{th}$ exterior power of $\Jac(f)$.  The invertibility of $\Jac(f)$ 
and functoriality of $\wedge^m(-)$ imply the invertibility of
$\wedge^m(\Jac(f))$. 
Then since $M(\theta)$ is also invertible, so is the tensor product
$\wedge^m(\Jac(f)) \otimes M(\theta)$,
and hence $\widetilde{\BBB}^{(m)}$ is another $K$-basis.
The last assertion of the proposition then follows, since
in the $m=1$ case, 
$$
\det( \Jac(f) \otimes M(\theta) )
= \det( \Jac(f) )^\ell \cdot \det( M(\theta) )^\ell
= J^\ell Q^\ell. \qedhere
$$
\end{proof}

We will show that
Theorem~\ref{main-result} follows from the next theorem.
\begin{theorem}
\label{linear-independence-theorem} 
Let $W$ be a complex reflection group
with homogeneous basic invariants $f_1,\ldots,f_\ell$,
and an index $i_0$ in $1,2, \ldots, \ell$ that satisfies
\begin{equation}
\label{regular-vector-assumption}
V^\reg \cap \bigcap_{i \neq i_0} f_i^{-1}\{0\} \neq \varnothing.
\end{equation}
Then for each $m=0,1,\ldots,\ell$, the following set gives a
$K$-vector space basis for $K \ot \wedge^m V^*  \ot V$:
\vspace{.5ex}
\begin{equation}
\label{more-general-putative-basis}
\BBB^{(m)} := 
 \Big\{ df_{I}\, d\theta_k\Big\}_{I \in \sbinom{[\ell] \setminus \{i_0\}}{m-1},
                              \, k \in [\ell]} 
\quad \sqcup \quad
 \Big\{ df_{I}\, \theta_k \Big\}_{I \in \sbinom{[\ell] \setminus \{i_0\}}{m},
                              \, k \in [\ell]}.
\end{equation}
\end{theorem}
\begin{proof}[Proof of Theorem~\ref{linear-independence-theorem}.]
There is nothing to prove in the case $m=0$.
We consider first the extreme case  $m=1$, then
the opposite extreme case $m = \ell$, and finally the intermediate cases with $2 \leq  m \leq \ell-1$.

\vskip.1in
\noindent
{\sf The case $m=1$.}
Note that the set $\BBB^{(1)}$ 
that we want to show is a $K$-basis for $K \ot V^* \ot V$,
$$
\BBB^{(1)} = \{ df_i\, \theta_k: i \in [\ell] \setminus \{i_0\}, k \in [\ell]\} 
                      \quad \sqcup \quad  \{ d\theta_k : k \in [\ell]\},
$$
has substantial overlap with the known $K$-basis 
$\widetilde{\BBB}^{(1)}$ for $K \ot V^* \ot V$ given in 
Lemma~\ref{more-obvious-basis-lemma},
$$
\widetilde{\BBB}^{(1)} 
  = \{ df_i \, \theta_k : i,k \in [\ell] \} 
  = \{ df_i \, \theta_k : i \in [\ell] \setminus \{i_0\}, k \in [\ell]\}
               \quad \sqcup \quad \{ df_{i_0} \theta_k : k \in [\ell]\}.
$$
Thus we need only show that when working in the quotient
of $K \ot V^* \ot V$ by the $K$-subspace spanned by
$$
\BBB^{(1)} \cap \widetilde{\BBB}^{(1)} 
= \{ df_i \, \theta_k: i \in [\ell] \setminus \{i_0\}, k \in [\ell] \}\, ,
$$
a {\it nonsingular} matrix in $K^{\ell \times \ell}$ expresses the images of the 
elements 
$$
\{ d\theta_k : k \in [\ell] \}=
\BBB^{(1)} \setminus \BBB^{(1)} \cap \widetilde{\BBB}^{(1)}
$$
uniquely in terms of the images of the elements
$$
\{ df_{i_0} \, \theta_k : k \in [\ell] \}=
\widetilde{\BBB}^{(1)}  \setminus \BBB^{(1)} \cap \widetilde{\BBB}^{(1)}\, .
$$
Here is how one produces this $\ell \times \ell$ matrix.  First
use Lemma~\ref{more-obvious-basis-lemma} to uniquely write
\begin{equation}
\label{d-thetas-in-terms-of-more-obvious-basis}
d \theta_k = \sum_{i,j \in [\ell]} r_{i,j,k}\ df_i \, \theta_j
\quad\text{ for each } k\in [\ell]\, ,
\end{equation}
with $r_{i,j,k}$ in $(JQ)^{-\ell}S$.
Then the matrix in $K^{\ell \times \ell}$ that we wish to show is 
nonsingular is $(r_{i_0,j,k})_{j,k \in [\ell]}$. 
\vspace{1ex}\rule[-1ex]{0ex}{1ex}

To this end, apply to \eqref{d-thetas-in-terms-of-more-obvious-basis} the map $E \otimes \one_V$ from 
Definition~\ref{Euler-map-definitions}, giving
a system of equations in $K \otimes1 \otimes V$:
$$
e^*_k \ \theta_k = \sum_{i,j \in [\ell]} r_{i,j,k}\ 
\deg(f_i)\ f_i \, \theta_j
\quad\text{ for each } k\in [\ell].
$$
Since $\{ \theta_j \}_{j \in [\ell]}$ forms a $K$-basis for $K \otimes V$,
this gives a linear system in $K$: 
\begin{equation}
\label{K-linear-system}
e^*_k\ \delta_{j,k} = \sum_{i \in [\ell]} r_{i,j,k}\ \deg(f_i)\ f_i
\quad\text{ for each } j,k\in[\ell]\, ,
\end{equation}
where
$\delta_{j,k}$ denotes the Kronecker delta function.

To show that
$(r_{i_0,j,k})_{j,k \in [\ell]}$ in $K^{\ell \times \ell}$ is nonsingular,
we will evaluate each of its entries at a carefully chosen vector $v$.
By the hypothesis \eqref{regular-vector-assumption}, one can choose a 
vector $v$ in $V^\reg$ with the property that $f_i(v)=0$ 
for $i \neq i_0$. Since the coefficients $r_{i,j,k}$ lie in 
${(JQ)^{-\ell}} S$, and since $J, Q$ vanish nowhere on $V^\reg$,
one may evaluate the linear system
\eqref{K-linear-system} at $v$ to obtain a linear system over $\CC$:
\begin{equation}
\label{specialized-matrix-equation}
e^*_k\  \delta_{j,k} = r_{i_0,j,k}(v)\ \deg(f_{i_0})\ f_{i_0}(v)
\quad\text{ for each } j,k\in[\ell]\, .
\end{equation}
We claim $f_{i_0}(v) \neq 0$:  otherwise 
$f_i(v)=0$ for every $i$ in $[\ell]$, meaning $v$ is in
the common zero locus within $V$ of the homogeneous system of
parameters $f_1, \ldots, f_{\ell}$ in $\SSS$, forcing the
contradiction $v=0$.
As the coexponents $e^*_k$ are also nonzero, \eqref{specialized-matrix-equation}
shows that the specialized matrix $(r_{i_0,j,k}(v))_{j,k \in [\ell]}$ 
in $\CC^{\ell \times \ell}$ is diagonal with nonzero determinant.
Hence it is nonsingular, and so is the unspecialized matrix 
$(r_{i_0,j,k})_{j,k \in [\ell]}$ in $K^{\ell \times \ell}$, as desired.

\vskip.1in
\noindent
{\sf The case $m=\ell$.}
To show that
$
\BBB^{(\ell)} = \{ df_{[\ell]\setminus\{i_0\}}\, d\theta_k \}_{k \in [\ell]}
$
is $K$-linearly independent,
consider a dependence
$$
0 = \sum_{k \in [\ell]} c_k\ df_{[\ell]\setminus\{i_0\}}\, d\theta_k.
$$
Substitute the expressions for $d\theta_k$ 
from Equation~\eqref{d-thetas-in-terms-of-more-obvious-basis} to obtain
$$
0 = \sum_{i,j,k \in [\ell]} 
     c_k\ r_{i,j,k}\ df_{[\ell]\setminus\{i_0\}}\, df_i \, \theta_j 
= \sum_{j,k \in [\ell]} 
     c_k\ (-1)^{i_0}\ r_{i_0,j,k}\ df_{[\ell]}\, \theta_j\, .
$$
But $\{df_{[\ell]}\, \theta_j\}_{j \in [\ell]}$ is a $K$-basis
for $K \ot \wedge^{\ell}(V^*) \ot V$
by  Lemma~\ref{more-obvious-basis-lemma}, 
hence, 
$$
0 = \sum_{k \in [\ell]} 
     c_k\ (-1)^{i_0}\ r_{i_0,j,k}
\quad\text{ for each } j \in [\ell]\, .
$$
The matrix $(r_{i_0,j,k})_{j,k \in [\ell]}$ was already
shown nonsingular in the  $m=1$ case, and hence
$c_k=0$ for each $k$.

\vskip.1in
\noindent
{\sf The intermediate cases $2 \leq  m \leq  \ell-1$.}
To show that $\BBB^{(m)}$ is $K$-linearly independent, consider a dependence
\begin{equation}
\label{intermediate-m-dependence}
0 = \sum_{\substack{I \in \bbinom{[\ell] \setminus \{i_0\}}{m-1} \\
                     k \in [\ell]\rule{0ex}{2ex} }} 
c_{I,k}\ df_I \, d\theta_k
\quad + \quad
  \sum_{\substack{I \in \bbinom{[\ell] \setminus \{i_0\}}{m} \\
                     k \in [\ell]\rule{0ex}{2ex} }} 
c_{I,k}\ df_{I}\, \theta_k\, .
\end{equation}
It suffices to show all coefficients $c_{I,k}$ in the first sum vanish:
If so, then the second sum gives a dependence among a subset of the $K$-basis $\widetilde{\BBB}^{(m)}$ from 
Lemma~\ref{more-obvious-basis-lemma}, and hence
its coefficients $c_{I,k}$ must also vanish.
To this end, fix a subset $I_0 \in \sbinom{[\ell] \setminus \{i_0\}}{m-1}$
and consider its complementary subset within $[\ell] \setminus \{i_0\}$, 
namely,
$$
I_0^c:=[\ell] \setminus \{i_0\} \setminus I_0.
$$ 
Since 
$
|I_0^c|=(\ell-1)-(m-1)=\ell-m,
$
we note that
\begin{itemize}
\item
$I_0^c \cap I \neq \varnothing$ for each 
$I \subset [\ell] \setminus \{i_0\}$ with $|I|=m$, and
\item
$I_0^c \cap I \neq \varnothing$ for each
$I \subset [\ell] \setminus \{i_0\}$ with $|I|=m-1$ and $I \neq I_0$.
\end{itemize}
Consequently, multiplying both sides of 
Equation~\eqref{intermediate-m-dependence} by $df_{I_0^c}$ causes all terms in the second sum to vanish, 
as well as most of the terms in the first sum, leaving only
$$
0 = \sum_{k \in [\ell] } 
\pm c_{I_0,k}\ df_{[\ell] \setminus \{i_0\}}\, d\theta_k\, ,$$
with sign corresponding to that in 
$df_{I_0^c}\wedge df_{I_0}= \pm df_{[\ell] \setminus \{i_0\}}$.
But then by the case $m=\ell$ already proven, the
coefficients $c_{I_0,k}=0$ all vanish, and thus the
coefficients in the first sum of 
Equation~\eqref{intermediate-m-dependence} vanish, as desired.
This completes the proof of Theorem~\ref{linear-independence-theorem}.
\end{proof}

\section{Numerology of duality groups}
\label{numerology-works-section}

We now fix our focus on duality groups
and the candidate basis for $(\SSS\ot \wedge^m V^* \ot V)^W$
given in Theorem~\ref{main-result}. We check in this section
that these sets, comprising the putative basis,
have appropriate degree sum.  Let
\begin{equation}
\label{specific-candidate-basis}
\BBB^{(m)} = \Big\{ df_I \,\theta_k \Big\}_{_I \in \bbinom{[\ell-1]}{m},\, k \in [\ell]}
\ \sqcup \
\Big\{ df_I\, d\theta_k \Big\}_{I \in \bbinom{[\ell-1]}{m-1},\, k \in [\ell]} \ \quad\text{for}\quad 0 \leq m \leq \ell\, . 
\end{equation}
Here one interprets the second set as empty when $m=0$ and the first set as empty when $m=\ell$.
\begin{lemma}
\label{rightdegree}
For a duality group $W$, 
the sum of the degrees of elements in $\BBB^{(m)}$ above is
$$
\Delta(\wedge^m V^*\ot V)=(\ell-1)\sbinom{\ell-1}{m-1}N + \sbinom{\ell-1}{m}N^*.
$$
\end{lemma}

\begin{proof}
Using the shorthand
notation $e_I:=\sum_{i \in I} e_i$
for subsets $I \subset [\ell]$,
note that 
$$
\deg (df_I\, \theta_k) = e_I+e^*_k
\quad\quad\text{and}\quad\quad
\deg (df_I \, d\theta_k) = e_I+e^*_k-1,
$$
and therefore the sum of degrees for $\BBB^{(m)}$ is
$$
\begin{aligned} 
&\sum_{\substack{I \in \bbinom{[\ell-1]}{m},\\ k \in [\ell]
\rule{0ex}{2ex}}}
 (e_I + e^*_k)
 +
\sum_{\substack{I \in \bbinom{[\ell-1]}{m-1},\\k \in [\ell]
\rule{0ex}{2ex}}}
  (e_I + e^*_k-1) \\
&\qquad=
\sum_{k \in [\ell]}
  \left( \sum_{I \in  \bbinom{[\ell-1]}{m}} e^*_k
         +  \sum_{I \in  \bbinom{[\ell-1]}{m-1}} (e^*_k-1) \right)
+\sum_{k\in [\ell]} \left(
            \sum_{I \in  \bbinom{[\ell-1]}{m}} e_I
           + \sum_{I \in  \bbinom{[\ell-1]}{m-1}} e_I
        \right).
\end{aligned}
$$
Now the first sum over $k$ can be rewritten as 
$$ 
\sbinom{\ell-1}{m} N^* + \sbinom{\ell-1}{m-1}(N^*- \ell)
=\sbinom{\ell}{m} N^* - \ell\sbinom{\ell-1}{m-1},
$$
while the second sum over $k$ can be rewritten as
$$
\ell \left( 
               \sum_{i \in [\ell-1]} \sbinom{\ell-2}{m-1} e_i +
               \sum_{i \in [\ell-1]} \sbinom{\ell-2}{m-2} e_i
             \right) 
= \ell\sbinom{\ell-1}{m-1} ( N - (h-1) )
$$
since 
$\sum_{i \in [\ell-1]} e_i 
 = (\sum_{i \in [\ell]} e_i ) - e_\ell
 = N- (h-1)$
by Equation~\eqref{Shephard-Todd-sum-of-exponents-fact},
as $e_\ell+1=\deg(f_\ell)=h$ by definition.
Hence the degree sum is 
$$
\sbinom{\ell}{m} N^*  + \ell N\sbinom{\ell-1}{m-1} 
-\ell h \sbinom{\ell-1}{m-1}
=
 \sbinom{\ell-1}{m} N^*  + (\ell-1)\sbinom{\ell-1}{m-1}N
=\Delta(\wedge^m V^*\ot V),
$$
where the first equality used the duality group equation
$N + N^* = h \ell$ (see~(\ref{h-expressed-in-N-and-N*})).
\end{proof}


\section{Duality groups and 
proof of Theorem~\ref{main-result}}
\label{well-generated-section}
We now investigate differential derivations
invariant under duality groups
and prove Theorem~\ref{main-result}.
We combine the linear independence results from Section~\ref{independence-section} with the numerology of the last section.
We first check that 
the hypothesis~\eqref{regular-vector-assumption} in 
Theorem~\ref{linear-independence-theorem} holds for all duality groups when one
chooses the index $i_0=\ell$.
We emphasize that although both
Lemma~\ref{irreducible-implies-unique-lowest-coexponent}
and its consequence 
Corollary~\ref{well-generated-implies-unique-highest-degree}
below could easily be checked case-by-case, we give case-free proofs
so that Theorem~\ref{main-result} relies on no classification
of reflection groups.

\begin{lemma}
\label{irreducible-implies-unique-lowest-coexponent}
Irreducible complex reflection groups have exactly one
coexponent equal to $1$.
\end{lemma}
 
\begin{proof} 
Since $V$ is a nontrivial irreducible $W$-representation,
and since the polynomial ring $S=\Sym(V^*)$ carries
the trivial representation $\one$ in its degree zero component $S_0$
and the representation $V^*$ in its degree one component $S_1$,
Schur's Lemma implies
$$
\begin{array}{rclclclcl}
\dim ( S_0 \otimes  V )^W
&=&\dim( \one \otimes V)^W
&=&\dim V^W
&=& 0 \, ,
& &\\
\dim ( S_1 \otimes  V )^W
&=&\dim ( V^* \otimes  V )^W
&=&\dim \Hom_{\CC}( V, V )^W
&=&\dim \Hom_{\CC W}( V, V )
&=&1.
\end{array}
$$
Hence among the $S^W$-basis elements $\theta_1,\ldots,\theta_\ell$ for 
$( S \otimes  V )^W$, there must be none of degree zero, and
exactly one of degree one; in fact, the latter must be a multiple of 
the {\it Euler derivation} $\theta_E:=x_1 \ot y_1 + \cdots + x_\ell \ot y_\ell$.
\end{proof}

\begin{remark}
The above proof shows more generally that even for non-reflection (finite) groups $W$ acting
nontrivially and irreducibly on $V=\CC^\ell$, there will be, up to scaling, only
the Euler derivation $\theta_E$ as a $W$-invariant derivation in $(S \otimes V)^W$ of degree one.
\end{remark}

\begin{corollary}
\label{well-generated-implies-unique-highest-degree}
For any duality group $W$, 
there is a unique highest exponent $e_\ell$ and 
accompanying unique highest degree 
$h=\deg(f_\ell)=e_\ell+1$.
\end{corollary}

\begin{lemma} 
\label{well-generated-implies-h-regular}
Duality groups $W$ satisfy 
hypothesis \eqref{regular-vector-assumption} 
in Theorem~\ref{linear-independence-theorem}
with $i_0=\ell$, that is,
$$
V^\reg \ \cap\ \bigcap_{i=1}^{\ell-1} f_i^{-1}\{0\} \neq \varnothing.
$$
\end{lemma}
\begin{proof}
(cf. \cite[p.~4]{BessisR})
Springer~\cite[Prop.~3.2(i)]{Springer} showed that if $W$
is a complex reflection group  with 
basic invariants $f_1,\ldots,f_\ell$ and $\zeta$ is any primitive
$d^{th}$ root of unity in $\CC$, then
\begin{equation}
\label{Springer-fact}
\bigcap_{\substack{i=1,\ldots,\ell:\\ d \, \nmid\,  \deg(f_i)}}
f_i^{-1}\{0\} 
= \bigcup_{g \in W} \ker(\zeta \one_V - g).
\end{equation}
On the other hand, Lehrer and Michel~\cite[Thm.~1.2]{LehrerMichel} 
showed existence of $g$ in $W$ with
$\ker(\zeta \one_V - g) \cap V^\reg \neq \varnothing$
if and only if $d$ divides as many 
degrees $\deg(f_i)=e_i+1$ as codegrees $e^*_i-1$.
For a duality group $W$, the equations $e_i + e^*_i=h=\deg(f_\ell)$ 
imply $h$ 
divides as many degrees as codegrees.  Also,
Corollary~\ref{well-generated-implies-unique-highest-degree}
implies that $f_\ell$ is the {\it only} basic invariant of
degree $h$, so that \eqref{Springer-fact} gives the result.
\end{proof}

We can now deduce the two equivalent statements of our main result.
Recall that $R$ is the exterior subalgebra of 
the $W$-invariant forms 
$(\SSS \otimes \wedge V^*)^W
=\bigwedge_{\SSS^W}\{df_1,\ldots,df_{\ell}\}$
generated by all $df_i$ except for the last one $df_\ell$, that is,
$
R:=\bigwedge_{\SSS^W}
\{df_1,\ldots,df_{\ell-1}\}.
$

\vskip.1in
\noindent
{\bf Theorem~\ref{main-result}.}
{\it
For $W$ a duality (well-generated) complex reflection group,
$(\SSS \otimes \wedge V^* \otimes V)^W$ forms a
free $R$-module on $R$-basis 
$
\{\theta_1,\ldots,\theta_\ell, \,\, d\theta_1, \ldots, d \theta_\ell\}.
$ 
Equivalently, $(\SSS\ot \wedge^m V^* \ot V)^W$ for $0 \leq m \leq \ell$
has $\SSS^W$-basis
\begin{equation*}
\BBB^{(m)} = \Big\{ df_I \theta_k \Big\}_{I \in \bbinom{[\ell-1]}{m},\, k \in [\ell]}
\ \sqcup \
\Big\{ df_I d\theta_k \Big\}_{I \in \bbinom{[\ell-1]}{m-1},\, k \in [\ell]}\, .
\end{equation*}
}
\vskip.1in

\begin{proof}[Proof of Theorem~\ref{main-result}]
Lemma~\ref{well-generated-implies-h-regular} and
Theorem~\ref{linear-independence-theorem} imply
that $\BBB^{(m)}$ has nonsingular coefficient matrix.
Indeed, when $i_0=\ell$,
the set $\BBB^{(m)}$ in \eqref{more-general-putative-basis}
agrees with that in~\eqref{specific-candidate-basis}.
Note that this set has cardinality 
$$
\sbinom{\ell-1}{m-1} \ell + \sbinom{\ell-1}{m} \ell
= \sbinom{\ell}{m} \ell=\dim_K \left( K \ot \wedge^m V^*  \ot V \right)
=\rank_{\SSS^W}(\SSS\ot\wedge^mV^*\ot V)^W\, .
$$
Lemma~\ref{rightdegree} shows that  their degree sum
is appropriate, and the theorem then follows 
from Corollary~\ref{key-numerology-corollary}.
\end{proof}

\begin{remark}
Theorem~\ref{main-result} has an amusingly compact rephrasing:  defining
by convention $f_0:=1$,  the last assertion of the 
theorem is equivalent to the assertion
that $(\SSS\ot \wedge^m V^* \ot V)^W$ is a 
free $\SSS^W$-module on basis
\begin{equation}
\label{cute-rephrased-basis}
\{  \
df_{i_m} \cdots df_{i_2} \cdot d( f_{i_1} \theta_k ) 
:\ 
0 \leq i_1 < i_2 < \cdots <i_m \leq \ell-1
\text{ and } k \in [\ell]
\,  \}\, .
\end{equation}
The reason is that  the elements in \eqref{cute-rephrased-basis} with $i_1=0$
coincide with the basis elements $\{df_I d\theta_k\}$ in the second part of 
$\BBB^{(m)}$, while the elements in \eqref{cute-rephrased-basis} with $i_1 \geq 1$
are almost the same as the basis elements $\{df_I  \theta_k\}$ in the first part of 
$\BBB^{(m)}$, but differ from them by $f_{i_1}$ times an element in the second part
of $\BBB^{(m)}$.
\end{remark}

\section{Two dimensional reflection groups}
\label{2-dimensional-groups-section}

We now consider reflection groups acting on $2$-dimensional complex space,
i.e., the case when $\ell=2$.
We found in Theorem~\ref{main-result}
an $\SSS^W$-basis for $(\SSS \ot \wedge V^* \ot V)^W$
when $W$ is a duality group.
Here, we find a 
{\it different} choice of basis
that works for {\it any} 
rank $2$ complex reflection group $W$, 
duality or not.  
Suppose we have
basic derivations in $(\SSS \otimes V)^W = (\SSS\otimes 1 \otimes V)^W$
\begin{equation}
\label{rank-two-derivations}
\begin{aligned}
\theta_1&:=x_1 \ot 1 \otimes y_1 + x_2 \ot 1 \otimes y_2\ (=\theta_E)
\, , \qquad\text{and}\\
\theta_2&:=\ a\, \otimes 1\ot y_1 + \ b\, \ot 1\otimes y_2\, ,\\
\end{aligned}
\end{equation}
for some $a,b$ in $\CC[x_1,x_2]$.
With this indexing,  
$e_1^*=1$ and $e_2^*=\deg(a)=\deg(b)$.

\begin{theorem}
\label{rank-two-theorem}
For a complex reflection group $W$ acting on $\CC^2$,
and $0 \leq m \leq 2$, the following sets $\BBB^{(m)}$
give free $S^W$-bases for
$(\SSS \ot \wedge^m V^*  \ot V)^W$:
\begin{equation}
\begin{aligned}
\BBB^{(0)}&:=\{\theta_1, \theta_2 \} \\
\BBB^{(1)}&:=\{df_1\, \theta_1, df_2\, \theta_1, 
d\theta_1, d\theta_2 \} \\
\BBB^{(2)}&:=\{df_1\, d\theta_1, df_2\, d\theta_1 \}.
\end{aligned}
\end{equation}
\end{theorem}

\begin{proof}
The $m=0$ case is immediate.
For $m=1,2$, the basic derivations 
as in~\eqref{rank-two-derivations} give
$$
\Coef(\BBB^{(2)})
=\bordermatrix{~ 
& \text{\tiny $df_1 d\theta_1$} & 
\text{\tiny $df_2 d\theta_1$} 
\cr & & \cr
\text{\tiny $1\otimes x_1 \wedge x_2 \otimes y_1$} & 
-\frac{\del f_1}{\del x_2} &
-\frac{\del f_2}{\del x_2} \cr
  & &  \cr
\text{\tiny $1\otimes x_1 \wedge x_2 \otimes y_2$} & 
\frac{\del f_1}{\del x_1} & \frac{\del f_2}{\del x_1}\cr
  & & },
 \quad
\Coef(\BBB^{(1)})
= \bordermatrix{~ & 
\text{\tiny $df_1 \theta_1$} & 
\text{\tiny $df_2 \theta_1$} & 
\text{\tiny $d\theta_1$} & 
\text{\tiny $d\theta_2$} \cr
 & & & & \cr
\text{\tiny $1\otimes x_1 \otimes y_1$} & x_1\frac{\del f_1}{\del x_1} & x_1\frac{\del f_2}{\del x_1} & 1         & \frac{\del a}{\del x_1} \cr
 & & & & \cr
\text{\tiny $1 \otimes x_2 \otimes y_2$} &x_2\frac{\del f_1}{\del x_2} & x_2\frac{\del f_2}{\del x_2} & 1 & \frac{\del b}{\del x_2} \cr
 & & & & \cr
\text{\tiny $1\otimes x_1 \otimes y_2$} & x_2\frac{\del f_1}{\del x_1} & x_2\frac{\del f_2}{\del x_1} & 0 & \frac{\del b}{\del x_1} \cr
 & & & & \cr
\text{\tiny $1\otimes x_2 \otimes y_1$} & x_1\frac{\del f_1}{\del x_2} & x_1\frac{\del f_2}{\del x_2} & 0 & \frac{\del a}{\del x_2} \cr
  & & & & \cr}.
$$
One now computes that $\BBB^{(1)},  \BBB^{(2)}$ have the right degree sums
and satisfy the hypotheses of Corollary~\ref{key-numerology-corollary}:
$$
\begin{aligned}
\det \Coef(\BBB^{(2)})&= \tfrac{\del f_1}{\del x_1} 
\tfrac{\del f_2}{\del x_2}
-\tfrac{\del f_1}{\del x_2} \tfrac{\del f_2}{\del x_1} 
=\det \Jac(f_1,f_2)
=J
=J^{(\ell-1)\binom{\ell-1}{m-1}} Q^{\binom{\ell-1}{m}} \text{ for }\ell=2=m,\\
\det \Coef(\BBB^{(1)})
&= \left(
\tfrac{\del f_1}{\del x_1} \tfrac{\del f_2}{\del x_2}
-\tfrac{\del f_1}{\del x_2} \tfrac{\del f_2}{\del x_1}
\right)
\left( 
x_1 \left( x_1\tfrac{\del b}{\del x_1} + x_2 \tfrac{\del b}{\del x_2}\right)
- x_2 \left( x_2 \tfrac{\del a}{\del x_2} + x_1 \tfrac{\del a}{\del x_1}\right)
\right)
\rule{0ex}{5ex} \\
&= J\, (x_1 e_2^* b - x_2 e_2^* a) \\
&= e_2^*\, J \det(M(\theta_1,\theta_2))  
= e_2^*\, JQ
=e_2^*\, J^{(\ell-1)\binom{\ell-1}{m-1}} 
Q^{\binom{\ell-1}{m}} \text{ for }\ell=2, m=1.
\rule{0ex}{3ex} 
\qedhere
\end{aligned}
$$
\end{proof}

This gives an immediate Hilbert series corollary when $\ell=2$.
\begin{corollary}
For a complex reflection group $W$ acting on $\CC^2$,
$$
\frac{\Hilb( (\SSS \ot \wedge V^*  \ot V)^W; q,t)} 
{\Hilb( \SSS^W;q)}= 
 t^0(q+q^{e_2^*})
      +t^1(1+q^{e_2^*-1}+q^{e_1+1}+q^{e_2+1})
      +t^2(q^{e_1}+q^{e_2}).  
$$
\end{corollary}

In the case of a duality group $W$ with $\ell=2$, one can
check that this agrees with 
description \eqref{Hilbert-series-consequence},
bearing in mind that $e_1^*=1$ and $e_2^*+e_1=h=e_2+1$ 
with the above conventions.

\vspace{2ex}


\section{The reflection group $G_{31}$}
\label{G31-section}

The group $W=G_{31}$ is an irreducible complex reflection group 
of rank $4$ containing $60$ reflections, each of order $2$
(so $N^*=N=60$), although it is not the complexification of
a Coxeter group.  It is not a duality group; the exponents are $(7,11,19,23)$
and the coexponents are $(1,13,17,29)$.  
Using a computer to complete
a Molien-style summation as
in Lemma~\ref{standard-Molien-variant}
for $W=G_{31}$ (taking $U=V=\CC^4$),  one obtains
\vspace{1ex}
\begin{equation}
\label{G31-hilbert-series}
\frac{\Hilb( (\SSS \ot \wedge V^*  \ot V)^W; q,t)} 
{\Hilb( \SSS^W;q)}
=(1 + q^7 t)(1 + q^{11} t)(q+t)(1+q^{12})(1 + q^{19} t + q^{16} + q^{23} t)\, .
\end{equation}
It is not hard to see that this is inconsistent with a description of
$(\SSS \ot \wedge^m V^*  \ot V)^W$ exactly as in Theorem~\ref{main-result}.
However, rewriting 
the right side of \eqref{G31-hilbert-series} as
$$(1 + q^7 t)(1 + q^{11} t) \cdot
   (q+t)\left[(1 + q^{12} + q^{16} + q^{28}) + 
           (q^{19}t + q^{23} t+ q^{19+12}+q^{23+12}t) \right]
$$
suggests a modified statement.  
Let 
$$R':=\bigwedge_{\SSS^W}\{df_1,df_2\}$$
as a subalgebra of 
$(\SSS \ot \wedge V^*)^W=\bigwedge_{\SSS^W}\{df_1,df_2,df_3,df_4\}$.

\begin{theorem}
\label{G31-theorem}
For $W=G_{31}$, the $R'$-module 
$(\SSS \ot \wedge V^* \ot V)^W$ 
is free with $R'$-basis
$$
\big\{ \theta_i \,\, , d\theta_i \big\}_{i=1,2,3,4} 
\ \sqcup \ 
\Bigg\{ 
\begin{aligned}
df_3\, \theta_1,&  &df_4\, \theta_1,&  &df_3\, \theta_2,&   & df_4\, \theta_2,  \\
df_3\, d\theta_1,& &df_4\, d\theta_1,& &df_3\, d\theta_2,& &df_4\, d\theta_2 
\end{aligned}
\Bigg\}
\, .
$$
\end{theorem}
\begin{proof}
\vspace{1ex}
One can check that the elements listed above that lie in $\SSS\ot\wedge^m V^*\ot V$ 
have degrees adding to
$$
\Saitodegree(\wedge^m V^*\ot V)
=60\left( 3\binom{3}{m-1} + \binom{3}{m} \right), \quad\text{for}\quad
0\leq m\leq 4.
$$
Thus the theorem follows from Corollary~\ref{key-numerology-corollary}
after one checks that each matrix of coefficients 
for $0\leq m\leq 4$ is nonsingular.  
We did this in {\tt Mathematica}, using explicit choices of  
basic invariant polynomials $f_1,f_2,f_3,f_4$ of degrees $8,12,20,24$
and basic derivations $\theta_1,\theta_2,\theta_3,\theta_4$
of degrees $1,13,17,29$, constructed as prescribed by 
Orlik and Terao
(using Maschke's~\cite{Maschke} 
invariants $F_8, F_{12}, F_{20}$, 
with $F_{24}=\det \mathrm{Hessian}(F_8)$; 
see also Dimca and Sticlaru~\cite{DimcaSticlaru} and 
also~\cite[p.~285]{OrlikTerao}).
\end{proof}

\section{Proof of Theorem~\ref{generation-result}}
\label{poorly-generated-section}

We first recall the statement of the theorem.
\vskip.1in

\noindent
{\bf Theorem~\ref{generation-result}.}
{\it
For $W$ any complex reflection group,
$(\SSS \otimes \wedge V^* \otimes V)^W$ is generated as a
module over the exterior algebra 
$(\SSS \otimes \wedge V^*)^W=\bigwedge_{\SSS^W}\{df_1,\ldots,df_{\ell}\}$ by the $2\ell$ generators
$
\{\theta_1,\ldots,\theta_\ell, \,\, d\theta_1, \ldots, d \theta_\ell\}.
$ 
}

\begin{proof}[Proof]
The general statement follows from the case where $W$ is irreducible.
For irreducible $W$, we proceed case-by-case, taking
advantage of the fact that the irreducible non-duality 
(that is, not well-generated)
reflection groups fall into three camps:
\begin{itemize}
\item The 2-dimensional groups ($\ell=2$).
\item The exceptional group $G_{31}$ (with $\ell=4$).
\item The infinite family of monomial groups $G(r,p,\ell)$ for $1<p<r$.
\end{itemize}
Reflection groups of dimension $2$
were considered in Section~\ref{2-dimensional-groups-section}; Theorem~\ref{rank-two-theorem} gives a basis.
The group $G_{31}$ was considered in Section~\ref{G31-section};
Theorem~\ref{G31-theorem} gives a basis.
The groups $G(r,p,\ell)$ are considered in the appendix,
as some direct computation is required to prove
the pattern in this general case; 
Theorem~\ref{monomial-group-theorem} gives a basis.
In each case, we provided an explicit $S^W$-module basis
for $(\SSS \otimes \wedge V^m \otimes V)^W$ 
whose elements all have
either the form $df_I \theta_k$ or $df_I d\theta_k$ for various subsets $I \subset [\ell]$
and $k$ in $[\ell]$.
\end{proof}

\section{Remarks and questions}
\label{questions-section}


\subsection*{What about $U=\wedge^k V$?}

One might wonder whether for complex reflection groups $W$,
or even just duality groups, one can
factor the Hilbert series more generally for 
$
\left( \SSS \otimes \wedge V^* \otimes \wedge^k V \right)^W
$
when $k$ takes values besides $k=0,1$.
One can manipulate Molien-style computations using 
this consequence of Lemma~\ref{standard-Molien-variant}:
$$
\begin{aligned}
\Hilb\left( \left( \SSS \otimes 
\wedge V^* \otimes \wedge V\right)^W; 
\ q, t, u \right)
&:= \sum_{i,j,k} \left( 
                 \dim \SSS_i \otimes \wedge^j V^* \otimes \wedge^k V
              \right)^W q^i\, t^j\, u^k\\
&=\tfrac{1}{|W|\rule{0ex}{1.5ex}} \sum_{w \in W} \frac{\det(1+uw^{-1})\det(1+tw)}{\det(1-qw)}\ .
\end{aligned}
$$
Things seem not to factor so nicely unless $k \in 
\{0,1,\ell-1,\ell\}$, but at least we
have a reciprocity:
\begin{prop}
\label{perfectpairing}
Let $W$ be a complex reflection group and set 
$$
\tau(q,t,u) := \Hilb \big( (\SSS \otimes \wedge V^* \ot \wedge V)^W; 
q,t, u\big)\, .
$$
Then $\tau$ satisfies the reciprocity
$$
\tau(q,t,u) =  t^{\ell}\, u^{\ell}\, \tau(q,t^{-1}, u^{-1})\, .
$$
\end{prop}
\begin{proof}
Let $\CC_{\det}$ be a $1$-dimensional $W$-module carrying the
determinant character of $W$ acting on $V$,
and likewise for $\CC_{\det^{-1}}$.
The $W$-equivariant perfect pairings
$$
\wedge^j\, V^*\otimes \wedge^{\ell-j}\, V^*
\longrightarrow \wedge^{\ell}\, V^*
\cong \CC_{\det^{-1}}
\quad\text{and}\quad
\wedge^k V\otimes \wedge^{\ell-k}\, V
\longrightarrow \wedge^{\ell}\, V
\cong \CC_{\det}
$$
imply that
$$
\SSS\ot \wedge^j \, V^* \ot \CC_{\det^{-1}}
\cong
\SSS\ot \wedge^{\ell-j}\, V\, 
\quad\text{and}\quad
\wedge^k  V \ot \CC_{\det}
\cong
\wedge^{\ell-k}\, V^*\, 
$$
as $W$-modules
(see~\cite{Shepler}, proof of Corollary 4),
since $V\cong V^{**}$ as $W$-modules.
The result then follows from the isomorphisms
of $W$-modules
$$
\SSS\ot \wedge^j\, V^* \ot \wedge^k\, V
\cong
(\SSS\ot \wedge^j \, V^* \ot \CC_{\det^{-1}})
\ot
(\wedge^k\,  V \ot \CC_{\det})
\cong
(\SSS\ot \wedge^{\ell-j}\, V)
\ot (\wedge^{\ell-k}\, V^*)\, .\qedhere
$$
\end{proof}

A similar argument confirms the following.
\begin{prop}
Let $W$ be a complex reflection group and
set 
$$\tau(\chi, q,t,u)
:=
\Hilb \big( 
(\SSS \otimes \wedge V^* \ot \wedge V\ot \CC_\chi)^W; q,t, u \big) 
=\tfrac{1}{|W|\rule{0ex}{1.5ex}} 
\sum_{w \in W}\chi^{-1}(w)\, \frac{\det(1+uw^{-1})\det(1+tw)}{\det(1-qw)}\, 
$$
for any character $\chi: W \rightarrow \CC^*$
afforded by a $1$-dimensional $W$-module $\CC_\chi$.
Then $\tau$ satisfies the reciprocity
$$
\tau(\chi; q, t, u) =
t^{\ell}\, u^{\ell}\, \tau(\chi; q, t^{-1},u^{-1})\, .
$$
\end{prop}

\begin{remark}
The last two results generalize an observation 
for real reflection groups
from~\cite[Eqn.~(1.24)]{GNS}.
\end{remark}

\begin{example}
For the Weyl group $W=W(F_4)$, with exponents $(1,5,7,11)$,  and $V^* \cong V$, 
a computation in {\tt Mathematica} gives
\begin{small}
$$
\begin{aligned}
&
 \Hilb\left( \left( 
\SSS \otimes \wedge V \otimes \wedge V \right)^W; q, t, u \right) / \
\Hilb(\SSS^W,q)
\hfill \\
& = \ \ \, u^0 (1 + q t)(1 + q^5 t)(1 + q^7 t)(1 + q^{11} t)  \\
 &\quad + u^1 (q+t)(1 + q^4 + q^6+q^{10})
      (1 + qt)(1 + q^5 t)(1 + q^7 t) \\
 &\quad + u^2 (q + t)(1 + q t) (1+q^4)
\Big( (q^5+q^7-q^9+q^{11}+q^{13})(1+t^2)
+ ( 1+q^6+q^8+q^{10}+q^{12}+q^{18})t \Big) \\
&\quad + u^3 (1+qt)(1 + q^4+ q^6+q^{10})(q+t)(q^5 + t)(q^7 + t) \\
&\quad + u^4 (q + t)(q^5 + t)(q^7 + t)(q^{11} + t) \, .
\end{aligned}
$$
\end{small}
The coefficient of $u^2$ does not seem to factor further,
but Proposition~\ref{perfectpairing} explains 
the duality between the coefficients of 
$u^k$ and of $u^{\ell-k}$.
\end{example}

\section{Appendix: The case of $G(r,p,\ell)$}
\label{appendix-section}
The Shephard and Todd
infinite family of reflection groups
$G(r,p,\ell)$ includes the Weyl groups 
of types $B_\ell$,  $D_\ell$, the dihedral groups, and symmetric groups.
To define these groups, fix an integer $r\geq 1$.  Then $G(r,1,\ell)$
is the set of $\ell \times \ell$
monomial matrices (i.e., matrices with a single 
nonzero entry in each row and column) whose nonzero
entries are complex $r$-th roots of unity.
The group $G(r,1,\ell)$ is the wreath product
of the symmetric group of order $\ell !$ and a cyclic group:
$$G(r,1,\ell)\cong \text{Sym}_\ell \wr \mathbb{Z}/r\mathbb{Z}  \cong \text{Sym}_\ell \ltimes(\mathbb{Z}/r\mathbb{Z})^\ell
\, .$$
Each group $G(r,1,\ell)$ acts on $V=\CC^\ell$ as a 
reflection group generated
by complex reflections of order $2$ and order $r$.
In fact, the group $G(r,1,\ell)$ is the symmetry group of
the complex cross-polytope in $\CC^\ell$, a {\it regular complex polytope} as studied
by Shephard~\cite{Shephard} and Coxeter~\cite{Coxeter}.

For integers $p \geq 1$ dividing $r$,
the group $G(r,p,\ell)$ consists of those matrices in $G(r,1,\ell)$ 
whose product of nonzero entries is an $(r/p)$-th root-of-unity.
Both $G(r,1,\ell)$ and $G(r,r,\ell)$ are duality groups,
and hence covered by Theorem~\ref{main-result}.
When $1<p<r$, the group $G(r,p,\ell)$ is a nonduality-group.

We record here a convenient choice of basic invariant polynomials, derivations for $G(r,p,\ell)$.

\begin{proposition}
Let $W=G(r,p,\ell)$ with $1\leq p < r$ and $p$ dividing $r$.
One may choose 
basic $W$-invariant polynomials $\{f_i\}_{i=1}^\ell$ in $S$ and 
derivations $\{\theta_i\}_{i=1}^\ell$ in $S\ot 1 \ot V$ as follows:
$$
\begin{aligned}
f_k&=x_1^{rk}+\cdots+x_\ell^{rk}\ \  \text{ for }k=1,2,\ldots,\ell-1, \text{ and }
f_\ell=(x_1 \cdots x_\ell)^{\frac{r}{p}},
 & \\
\theta_k &= x_1^{(k-1)r+1} \otimes 1 \ot y_1 + \cdots + x_\ell^{(k-1)r+1} \otimes 1 \ot y_\ell
\ \ \text{ for }k=1,2,\ldots,\ell\, .
\end{aligned}
$$
In particular, $W$ has 
$$
\begin{aligned}
&\text{exponents}\quad&
&(e_1,\ldots,e_{\ell-1},e_\ell)
=(r-1,2r-1,\ldots,(\ell-1)r-1,\tfrac{\ell r}{p}-1), \\
&\text{coexponents}\quad&
&(e_1^*,\ldots,e_\ell^*)
=(1,r+1,2r+1,\ldots,(\ell-1)r+1),\\
&\text{number of reflections}\quad& 
&N
=\sbinom{\ell}{2}r+\ell\left(\tfrac{r}{p}-1\right),
\quad\text{and}\\
&\text{number of hyperplanes}\quad& 
&N^*
=\sbinom{\ell}{2}r+\ell\, .
\end{aligned}
$$
\end{proposition}
\begin{proof}
The $W$-invariant derivations $\{\theta_i\}_{i=1}^\ell$ above are the 
usual choice, for example, as in Orlik and Terao~\cite[Prop.~6.77]{OrlikTerao}.
(Or use the $m=0$ case of
Corollary~\ref{key-numerology-corollary} to verify the $\theta_i$ are basic derivations, 
since their associated coefficient matrix is an easy variant of a
Vandermonde matrix.)  
The $W$-invariant polynomials $\{f_i\}_{i=1}^\ell$ above are closely related to a more usual choice $\{f_i'\}_{i=1}^\ell$ of basic $W$-invariants 
$$
e_1(x_1^r,\ldots,x_\ell^r), \,
e_2(x_1^r,\ldots,x_\ell^r), \,
\ldots, \,
e_{\ell-1}(x_1^r,\ldots,x_\ell^r), \,
(x_1\cdots x_\ell)^{\frac{r}{p}},
$$
in which $e_k(x_1,\ldots,x_\ell)$ is the $k^{th}$ elementary symmetric polynomial in $x_1,\ldots,x_\ell$,
the sum of all square-free monomials of degree $k$;
e.g., see Smith~\cite[\S7.4, Ex.~1]{Smith}.
  However, 
$\{f_i\}_{i=1}^{\ell-1}$ and $\{f'_i\}_{i=1}^{\ell-1}$ 
generate the same subalgebra of polynomials
when working over a field of characteristic zero 
because the collection of power sums and 
elementary symmetric functions
can be expressed as polynomials in each other; 
e.g., see~\cite[Thm.~7.4.4, Cor.~7.7.2, Prop.~7.7.6]{Stanley}.
\end{proof}

\begin{theorem}
\label{monomial-group-theorem} 
Let $W=G(r,p,\ell)$ with $1\leq p<r$ and $p$ dividing $r$.  
Then for $1 \leq m \leq \ell$,
the set
$$
\begin{aligned}
\BBB^{(\ell,m)}
&=\left\{df_I\, \theta_k\right\}_{ 
I \in \bbinom{[\ell]}{m}, \ k \in [\ell-m]} \
\ \  \sqcup\ \
\left\{df_{I}\, d\theta_k\right\}_{ 
I \in \bbinom{[\ell]}{m-1},\ k \in [\ell-m+1]}
\end{aligned}
$$
gives a free basis for the $\SSS^W$-module of invariants
$\big(\SSS\ot \wedge^m V^* \ot V)^W$.
\end{theorem}

\begin{proof}
We will apply Corollary~\ref{key-numerology-corollary}
to $\BBB^{(\ell,m)}$.  There are several steps.
\vskip.1in
\noindent
{\sf Step 1.}
First note that $\BBB^{(\ell,m)}$ has the correct cardinality:
$$
\begin{aligned}
|\BBB^{(\ell,m)}|
&=(\ell-m)\sbinom{\ell}{m} + (\ell-m+1)\sbinom{\ell}{m-1}\\
&=(\ell-m)\sbinom{\ell}{m} + m\sbinom{\ell}{m}\\
&=\ell\sbinom{\ell}{m}
=\rank_{\SSS^W}(\SSS\ot\wedge V^*\wedge V)^W\, .
\end{aligned}
$$

\vskip.1in
\noindent
{\sf Step 2.}
We check that the sum of the degrees of the elements in $\BBB^{(\ell,m)}$ is $\Delta(\wedge^mV^*\ot V)$,
which is straightforward albeit tedious.
Again using the shorthand notation $e_I:=\sum_{i \in I} e_i$ for $I \subset [\ell]$,
this sum is
$$
\sum_{ \substack{I \in \bbinom{[\ell]}{m} \\ 
k\in[\ell-m]\rule{0ex}{2ex}}} 
(e_I+e_k^*)
 + \sum_{ \substack{I\in \bbinom{[\ell]}{m-1} \\ 
k\in[\ell-m+1]\rule{0ex}{2ex}}} (e_{I}+e_k^*-1)\, ,
 $$
 which one can rewrite as
$$
\begin{aligned}
&\sbinom{\ell}{m} \sum_{k \in [\ell-m]}e_k^* \ + 
     \sbinom{\ell}{m-1} \sum_{k \in [\ell-m+1]}(e_k^*-1)\ + \\
& \qquad       (\ell-m)\left(  \sum_{I \in \bbinom{[\ell-1]}{m}} e_I 
                                      +  \sum_{I \in \bbinom{[\ell-1]}{m-1}} \left( e_I + e_\ell\right) \right) +
          (\ell-m+1)\left(  \sum_{I \in \bbinom{[\ell-1]}{m-1}} e_{I} 
                                      +  \sum_{I \in \bbinom{[\ell-1]}{m-2}} \left( e_{I} + e_\ell\right) \right)\, .
\end{aligned}
$$
Bearing in mind that $e_i=ir-1$ for $1 \leq i \leq \ell-1$, 
we employ a shorthand notation
\begin{equation}
\label{sum-shorthand}
g(n,m):=\sum_{
\substack{I \in \bbinom{[k]}{m} \\ i \in I\rule{0ex}{1.5ex} 
}}e_i 
=\sum_{\substack{I \in \bbinom{[k]}{m} \\ i \in I\rule{0ex}{1.5ex} 
} } (ir-1) 
=\sbinom{k-1}{m-1} \sum_{i \in I} (ir-1) 
=\sbinom{k-1}{m-1} \left(r \sbinom{k+1}{2}-k \right)
\end{equation}
to rewrite the degree sum
as
$$
\begin{aligned}
&\sbinom{\ell}{m} \left( r \sbinom{\ell-m}{2} + \ell-m \right) +
    \sbinom{\ell}{m-1} \left( r \sbinom{\ell-m+1}{2} \right) +\\
& \qquad \qquad     (\ell-m) \left( g(\ell-1,m) + g(\ell-1,m-1) + \sbinom{\ell-1}{m-1}e_\ell \right) + \\
 &    \qquad \qquad\qquad    (\ell-m+1) \left( g(\ell-1,m-1) + g(\ell-1,m-2) + \sbinom{\ell-1}{m-2}e_\ell \right) \, .
\end{aligned}
$$
(Here, we use the fact that  $e_i^*=(i-1)r+1$ for all $i$.)
Finally, substituting in the right side 
of~\eqref{sum-shorthand} for all $g(n,m)$,
and $\frac{\ell r}{p}-1$ for $e_{\ell}$, we obtain
$$
\begin{aligned} 
(\ell-1)\tbinom{\ell-1}{m-1}
&
\left( \tbinom{\ell}{2}r+\ell \left(\tfrac{\ell r}{p}-1\right) \right)
+ \tbinom{\ell-1}{m}\left( \tbinom{\ell}{2}r+\ell\right)\\
&=(\ell-1)\tbinom{\ell-1}{m-1} N+ \tbinom{\ell-1}{m}N^* 
=\Delta(\wedge^m V \ot V^*). 
\end{aligned}
$$
The first equality here was checked by hand and 
corroborated in computer algebra packages.
\vskip.1in
\noindent
{\sf Step 3.}
At this stage, to apply Corollary~\ref{key-numerology-corollary}, 
we need only show that the
set $\BBB^{(\ell,m)}$ is $K$-linearly independent in $K \otimes \wedge^m V^* \otimes V$.
We will use this to reduce to the case where $p=1$, that is, $W=G(r,1,\ell)$.

Note that the formulas for $\theta_k, d\theta_k, df_k$ in $W=G(r,p,\ell)$ depend on the parameter $p$ in only one place, namely, in the definition of $df_\ell$:
\begin{equation}
\label{differential-formulas}
\begin{array}{rcll}
df_k&=&kr \sum_{j=1}^\ell x_j^{kr-1}  \otimes x_j &\text{ for }1 \leq k \leq \ell-1\, ,\\
\rule{0ex}{3ex}
df_\ell&=& \frac{r}{p} (x_1 \cdots x_\ell)^{\frac{r}{p}}  \sum_{j=1}^\ell x_i^{-1}  \otimes x_i\, , & \\
& & & \\
\theta_k&=&\sum_{j=1}^\ell x_j^{(k-1)r+1} \otimes 1 \otimes y_j &\text{ for }1 \leq k \leq \ell \, ,\\
\rule{0ex}{3ex}
d\theta_k&=&((k-1)r+1) \sum_{j=1}^\ell x_j^{(k-1)r} \otimes x_j \otimes y_j &\text{ for }1 \leq k \leq \ell.
\end{array}
\end{equation}

In checking whether the elements of 
$\BBB^{(\ell,m)}$ are $K$-linearly independent, we
are free to scale them by elements of the (rational function) field 
$K=\CC(x_1,\ldots,x_\ell)$. 
Hence, in \eqref{differential-formulas}, 
we may divide each element $d\theta_k$
by $(k-1)r+1$ for $1 \leq k \leq \ell$,  
we may also divide each element $df_k$
by $kr$ for $1 \leq k \leq \ell-1$,  and
lastly we may divide $df_\ell$ by 
$\frac{r}{p} (x_1 \cdots x_\ell)^{\frac{r}{p}}$.  
Hence Theorem~\ref{monomial-group-theorem} is equivalent to 
asserting $K$-linearly independence of the set 
$$
\BBB^{(\ell,m)}
=\left\{df_I \, \theta_k\right\}_{I \in \bbinom{[\ell]}{m},\ k \in [\ell-m]}
   \quad \sqcup \quad
 \left\{df_{I} \, d\theta_k\right\}_{I \in \bbinom{[\ell]}{m-1},\ k \in [\ell-m+1]}
$$
for $\theta_k, d\theta_k,df_k$  for  $1 \leq k \leq \ell$ redefined
(from~\ref{differential-formulas}) to
give a simple and uniform family:
$$
\begin{array}{rcl}
df_k&:=& \sum_{j=1}^\ell x_j^{(k-1)r-1} \otimes x_j, \\
\rule{0ex}{3ex}
\theta_k&:=& \sum_{j=1}^\ell x_j^{(k-1)r+1} \otimes 1 \otimes y_j,\\
\rule{0ex}{3ex}
d\theta_k&:=& \sum_{j=1}^\ell x_j^{(k-1)r} \otimes x_j \otimes y_j.
\end{array}
$$
Note that we have also employed a cyclic shift of the indexing, that is,
the old $df_\ell$ has been replaced by the new $df_1$, the old
$df_1$ by the new $df_2$, etc.

This new $K$-linear independence assertion does not involve the parameter $p$.   
Therefore Theorem~\ref{monomial-group-theorem} 
for $W=G(r,p,\ell)$ with $1 \leq p < r$ follows upon
proving it for $W=G(r,1,\ell)$, that is, with $p=1$.

\vskip.1in
\noindent
{\sf Step 4.}
We rescale the $K$-basis elements
$\{1 \otimes dx_I \otimes y_k\}_{I \in \bbinom{\ell}{m}, k \in [\ell]}$ in $K \otimes \wedge^m V^* \otimes V$
as follows:
$$
dx_I \otimes y_k \longmapsto
      x_k^{-1} x_I \otimes dx_I \otimes y_k
$$
(recall that 
$dx_I\ot y_k
=1\ot x_{i_1}\wedge\cdots\wedge x_{i_m}\ot y_k$
for $I=\{i_1 < \cdots< i_m\}$).
Then Theorem~\ref{monomial-group-theorem} 
is equivalent to the assertion that $\BBB^{(\ell,m)}$ is $K$-linearly independent
after redefining 
\begin{equation}
\label{last-redefinition-of-invariants}
\begin{array}{rlll}
df_k &:=\sum_{j=1}^\ell x_j^{(k-1)r} \otimes x_j &=\sum_{j=1}^\ell z_j^{k-1} \otimes x_j, 
\\
\theta_k&:=\sum_{j=1}^\ell x_j^{(k-1)r}\otimes 1\otimes  y_j &=\sum_{j=1}^\ell z_j^{k-1}\otimes 1 \otimes y_j, 
\rule{0ex}{3.5ex}\\
d\theta_k&:=\sum_{j=1}^\ell x_j^{(k-1)r} \otimes x_j \otimes y_j &=\sum_{j=1}^\ell z_j^{k-1} \otimes x_j \otimes y_j,
\rule{0ex}{3.5ex}\\
\end{array}
\end{equation}
for $k=1,2,\ldots,\ell$,
where we have set $z_j:=x_j^r$ in $K$.

\vskip.1in
\noindent
{\sf Step 5.}
Consider the matrix $B^{(\ell,m)}$ with entries in $\CC(z_1,\ldots,z_\ell)$
whose columns express each element of $\BBB^{(\ell,m)}$,
defined via \eqref{last-redefinition-of-invariants},
in terms of $\{ dx_I \otimes y_k\}$ for
$I \in \sbinom{[\ell]}{m}$ and $k \in [\ell]$.
Since each $d\theta_k$ is a $K$-linear combination of
terms of the form $1 \otimes x_j \otimes y_j$, the
expansion of each $df_{I}\, d\theta_k$ in $\BBB^{(\ell,m)}$
has nonzero coefficient of $dx_I \otimes y_k$ only when $k \in I$.
This leads to a block upper-triangular decomposition:
$$
B^{(\ell,m)} \quad =\quad
\bordermatrix{~      
& \text{\tiny $\{df_{I}\, d\theta_k \} $} 
&  \text{\tiny $\{df_{I}\, \theta_k \}$}  \cr
& & \cr
 \text{\tiny$\{ dx_I \otimes y_k:  k \in I \} $} & C^{(\ell,m)} & * \cr
 & & \cr
 \text{\tiny $\{ dx_I \otimes y_k: k \not\in I \}$} & 0 & D^{(\ell,m)} \cr
 }.
$$
By convention here, $C^{(\ell,0)}$ and $D^{(\ell,\ell)}$ are $0 \times 0$ matrices.

Thus it remains to show that $C^{(\ell,m)}$ and $D^{(\ell,m)}$
are invertible.  
We reduce this to showing invertibility of only $D^{(\ell,m)}$,
since we claim that for $m=1,2,\ldots,\ell$, the
matrices $C^{(\ell,m)}$ and $D^{(\ell,m-1)}$ differ only by row-scalings.
To justify this claim, consider a pair $(J,j_0)$ with 
$J=\{j_1 <\cdots<j_{m-1}\}$ for an $(m-1)$-subset of $[\ell]$
and $j \in [\ell-m+1]$.  Then  $(J,j_0)$
indexes both a column in $C^{(\ell,m)}$, 
that lists the expansion coefficients in 
\begin{equation}
\label{C-expansion}
\begin{aligned}
df_{J}\, d\theta_{j_0}
&=\left( \sum_i z_i^{j_1-1} \otimes x_i\right)
\cdots
\left( \sum_i z_i^{j_{m-1}-1} \otimes x_i \right)
\left( \sum_i z_i^{j_0-1} \otimes x_i \otimes y_i\right)\\
&=\sum_{(i_1,\ldots,i_{m-1},i_0)} 
     z_{i_1}^{j_1-1} \cdots  z_{i_{m-1}}^{j_{m-1}-1}  z_{i_0}^{j_0-1} \otimes
       x_{i_1} \wedge \cdots \wedge x_{i_{m-1}} \wedge x_{i_0} \otimes y_{i_0}\, ,
\end{aligned}
\end{equation}
and a column in $D^{(\ell,m-1)}$, that lists the expansion coefficients in 
\begin{equation}
\label{D-expansion}
\begin{aligned}
df_{J} \, \theta_{j_0}
&=\left( \sum_i z_i^{j_1-1} \otimes x_i\right)
\cdots
\left( \sum_i z_i^{j_{m-1}-1} \otimes x_i \right)
\left( \sum_i z_i^{j_0-1} \otimes 1\otimes y_i\right)\\
&=\sum_{(i_1,\ldots,i_{m-1},i_0)} 
     z_{i_1}^{j_1-1} \cdots  z_{i_{m-1}}^{j_{m-1}-1}\,  
    z_{i_0}^{j_0-1} \otimes
       x_{i_1} \wedge \cdots \wedge x_{i_{m-1}} \otimes y_{i_0}\, .
\end{aligned}
\end{equation}
On the other hand, a pair $(I,k)$ where $I$ is an $m$-subset of $[\ell]$ and $k \in I$
will index both a row for $dx_I \otimes y_k$ in $C^{(\ell,m)}$
and a row for $dx_{I \setminus \{k\}} \otimes y_k$ in  $D^{(\ell,m-1)}$.
If one assumes that $I=\{i_0,i_1,\ldots,i_{m-1}\}$ in the two expansions \eqref{C-expansion},
\eqref{D-expansion} above, then one can see that these two rows will differ by
a sign;  this sign is the product of the signs of two permutations,
namely, those permutations that sort the ordered sequences
$(i_1,\ldots,i_{m-1},i_0)$ and $(i_1,\ldots,i_{m-1})$ into the 
usual integer orders on $I$ and $I \setminus \{k\}$, respectively.

\vskip.1in
\noindent
{\sf Step 6.}
It remains to show invertibility for $0 \leq m \leq \ell$ of the square matrix $D^{(\ell,m)}$,
that is, the submatrix whose columns give the expansion coefficients in $\CC(z_1,\ldots,z_\ell)$ for each element of
$$
\{df_I\, \theta_k:  I \in \bbinom{[\ell]}{m},\, k \in [\ell-m] \}
$$
in terms of the basis elements 
$$
\{ dx_I \otimes y_k : I \in \bbinom{[\ell]}{m}, k \not\in I \},
$$
ignoring coefficients on all other $K$-basis elements $dx_I \otimes y_k$.

In fact, we will show that $\det D^{(\ell,m)}$ has coefficient $\pm 1$ on its {\it lexicographically-largest} 
monomial, that is, the monomial
$
z_\ell^{a_\ell} z_{\ell-1}^{a_{\ell-1}} \cdots z_2^{a_2} z_1^{a_1}
$  
that achieves the maximum exponent $a_\ell$, and 
among all such monomials with maximum $a_\ell$,
also maximizes $a_{\ell-1}$, and so on.

We argue via induction on $\ell$
by considering the following block decomposition of $D^{(\ell,m)}$:
$$
\bordermatrix{~      
   & \text{\tiny $\{ df_I \theta_k: \ell \in I \}$}  
&  \text{\tiny $\{df_I \theta_{\ell-m}: \ell \not\in I\} $} 
& \text{\tiny $\{df_I \theta_k: \ell \not\in I, k \neq \ell-m\}$} \cr
 \text{\tiny $\{ dx_I \otimes y_k: \ell \in I\}$} &
 \alpha & * & * \cr
\text{\tiny $\{ dx_I \otimes y_\ell \}$}   
         &  \delta  & 
 \beta & * \cr
 \text{\tiny $\{dx_I \otimes y_k: \ell \not\in I \sqcup\{k\} \}$} 
& \phi & \epsilon & \gamma \cr }.
$$
We note two degenerate cases when $m=\ell$ or  $m=\ell-1$:  if $m=\ell$
then $D^{(\ell,\ell)}$ is $0 \times 0$, as pointed out in Step 5, leaving nothing to prove;
if $m=\ell-1$, the sets indexing the rightmost block of columns and the bottommost block of rows are empty, so that
$
D^{(\ell,\ell-1)}=\left( \begin{smallmatrix}
\alpha & *\\
\delta & \beta
\end{smallmatrix} \right).
$

As all rows $dx_I \otimes y_k$ have $k \not\in I$,
the highest powers of $x_\ell$ in entries of $\det D^{(\ell,m)}$ are
\begin{itemize}
\item at most $z_\ell^{\ell-1}$ in the top block of rows $(\alpha,*,*)$,
and $z_\ell^{\ell-1}$ occurs only in the block $\alpha$,
\item at most $z_\ell^{\ell-1-m}$ in the middle rows $(\delta,\beta,*)$,
and $z_\ell^{\ell-1-m}$ occurs only in blocks $\delta, \beta$,
\item only $z_\ell^0$  in the bottom block of rows $(\phi,\epsilon,\gamma)$, that is, 
no $z_\ell$'s occur at all there.
\end{itemize}

We examine the terms in the permutation expansion of $\det D^{(\ell,m)}$ that
achieve the highest power of $z_\ell$.  Since $1 \leq m \leq \ell-1$, without loss of generality, 
$z_\ell^{\ell-1}$ is a strictly higher power than $z_\ell^{\ell-1-m}$, and thus these terms must use only
entries from the block $\alpha$ in the topmost block of rows $(\alpha,*,*)$, that is, they must be terms from
the product
$
\det \alpha \cdot \det  \left( \begin{smallmatrix}
\beta & *\\
\epsilon& \gamma
\end{smallmatrix} \right)
$.
In the degenerate case $m=\ell-1$, they must be terms from $\det \alpha \cdot \det \beta$.  In the nondegenerate
cases $1 \leq m \leq \ell-2$, 
they
must be terms from  $\det \alpha \cdot \det \beta \cdot \det \gamma$
since $z_\ell^{\ell-1-m}$ is a strictly higher power than $z_\ell^0$;
furthermore, these terms must always pick up entries from $\alpha$ 
divisible by $z_\ell^{\ell-1}$ and entries from $\beta$ divisible by
$z_\ell^{\ell-1-m}$.

Upon examining $\alpha, \beta, \gamma$ more closely, 
one finds that 
\begin{equation}
\label{closer-examination-of-alpha-beta-gamma}
\begin{aligned}
\gamma&=D^{(\ell-1,m)},\\
\beta&=z_\ell^{\ell-1-m} \cdot \wedge^m\, {\VDM}^{(\ell-1)} ,\\
\alpha&=z_\ell^{\ell-1} D^{(\ell-1,m-1)} +O(z_\ell^{\ell-2})
\end{aligned}
\end{equation}
where $\wedge^m A$ is the $m^{th}$ {\it exterior power} of the
matrix $A$, and ${\VDM}^{(n)}:=[ z_i^{j-1} ]_{i,j=1,2,\ldots,n}$
is an $n \times n$ {\it Vandermonde matrix}.
For example, when $\ell=4,m=2$, the matrix $\beta$ is
$$
\bordermatrix{~ & 
\text{\tiny $df_3 df_2 \theta_2$} & 
\text{\tiny $df_3 df_1 \theta_2$} & 
\text{\tiny $ df_2 df_1 \theta_2$} \cr
 & & & \cr
\text{\tiny $1\otimes x_3 \wedge x_2 \otimes y_4$} 
& (z_3^2 z_2^1-z_2^2 z_3^1)z_4^1 
& (z_3^2 z_2^0-z_2^2 z_3^0)z_4^1 & (z_3^1 z_2^0-z_2^1 z_3^0)z_4^1  \cr
 & & & \cr
\text{\tiny $1\otimes x_3 \wedge x_1 \otimes y_4$} 
& (z_3^2 z_1^1-z_1^2 z_3^1)z_4^1 
& (z_3^2 z_1^0-z_1^2 z_3^0)z_4^1 & (z_3^1 z_1^0-z_1^1 z_3^0)z_4^1 \cr
 & & & \cr
\text{\tiny $1\otimes x_2 \wedge x_1 \otimes y_4$} 
& (z_2^2 z_1^1-z_1^2 z_2^1)z_4^1 & (z_2^2 z_1^0-z_1^2 z_2^0)z_4^1 & (z_2^1 z_1^0-z_1^1 z_2^0)z_4^1
}
=z_4^{4-1-2} \cdot \wedge^2\, {\VDM}^{(3)},
$$
and the matrix $\alpha$ 
(with terms with the highest power $z_4^3=z_\ell^{\ell-1}$ underlined) is

\begin{tiny}
$$
\bordermatrix{~ & df_4 df_3 \theta_2 & df_4 df_2 \theta_2 
                  &  df_4 df_1 \theta_2 &  df_4 df_3 \theta_1 
                      &  df_4 df_2 \theta_1 &  df_4 df_1 \theta_1 \cr
 & & & & & & \cr
1\otimes x_4 \wedge x_3\otimes y_2 & (\underline{z_4^3 z_3^2} - z_4^2 z_3^3) z_2^1  &
                       (\underline{z_4^3 z_3^1} - z_4^1 z_3^3) z_2^1 &
                        (\underline{z_4^3 z_3^0} - z_4^0 z_3^3) z_2^1 &
                         (\underline{z_4^3 z_3^2} - z_4^2 z_3^3) z_2^0 &
                          (\underline{z_4^3 z_3^1} - z_4^1 z_3^3) z_2^0 &
                            (\underline{z_4^3 z_3^0} - z_4^0 z_3^3) z_2^0 \cr
 & & & & & & \cr
1\otimes x_4 \wedge x_3\otimes y_1 & (\underline{z_4^3 z_3^2} - z_4^2 z_3^3) z_1^1  &
                       (\underline{z_4^3 z_3^1} - z_4^1 z_3^3) z_1^1 &
                        (\underline{z_4^3 z_3^0} - z_4^0 z_3^3) z_1^1 &
                         (\underline{z_4^3 z_3^2} - z_4^2 z_3^3) z_1^0 &
                          (\underline{z_4^3 z_3^1} - z_4^1 z_3^3) z_1^0 &
                            (\underline{z_4^3 z_3^0} - z_4^0 z_3^3) z_1^0 \cr
 & & & & & & \cr
1\otimes x_4 \wedge x_2\otimes y_3 & (\underline{z_4^3 z_2^2} - z_4^2 z_2^3) z_3^1  &
                       (\underline{z_4^3 z_2^1} - z_4^1 z_2^3) z_3^1 &
                        (\underline{z_4^3 z_2^0} - z_4^0 z_2^3) z_3^1 &
                         (\underline{z_4^3 z_2^2} - z_4^2 z_2^3) z_3^0 &
                          (\underline{z_4^3 z_2^1} - z_4^1 z_2^3) z_3^0 &
                            (\underline{z_4^3 z_2^0} - z_4^0 z_2^3) z_3^0 \cr
 & & & & & & \cr
1\otimes x_4 \wedge x_2\otimes y_1 & (\underline{z_4^3 z_2^2} - z_4^2 z_2^3) z_1^1  &
                       (\underline{z_4^3 z_2^1} - z_4^1 z_2^3) z_1^1 &
                        (\underline{z_4^3 z_2^0} - z_4^0 z_2^3) z_1^1 &
                         (\underline{z_4^3 z_2^2} - z_4^2 z_2^3) z_1^0 &
                          (\underline{z_4^3 z_2^1} - z_4^1 z_2^3) z_1^0 &
                            (\underline{z_4^3 z_2^0} - z_4^0 z_2^3) z_1^0 \cr
 & & & & & & \cr
1\otimes x_4 \wedge x_1\otimes y_3 & (\underline{z_4^3 z_1^2} - z_4^2 z_1^3) z_3^1  &
                       (\underline{z_4^3 z_1^1} - z_4^1 z_1^3) z_3^1 &
                        (\underline{z_4^3 z_1^0} - z_4^0 z_1^3) z_3^1 &
                         (\underline{z_4^3 z_1^2} - z_4^2 z_1^3) z_3^0 &
                          (\underline{z_4^3 z_1^1} - z_4^1 z_1^3) z_3^0 &
                            (\underline{z_4^3 z_1^0} - z_4^0 z_1^3) z_3^0 \cr
 & & & & & & \cr
1\otimes x_4 \wedge x_1\otimes y_2 & (\underline{z_4^3 z_1^2} - z_4^2 z_1^3) z_2^1  &
                       (\underline{z_4^3 z_1^1} - z_4^1 z_1^3) z_2^1 &
                        (\underline{z_4^3 z_1^0} - z_4^0 z_1^3) z_2^1 &
                         (\underline{z_4^3 z_1^2} - z_4^2 z_1^3) z_2^0 &
                          (\underline{z_4^3 z_1^1} - z_4^1 z_1^3) z_2^0 &
                            (\underline{z_4^3 z_1^0} - z_4^0 z_1^3) z_2^0 
}.
$$
\end{tiny}
\vspace{1ex}

\noindent
Note that here $\alpha=z_4^{4-1} \cdot D^{(3,1)} + O(z_4^{4-2})$, as asserted in \eqref{closer-examination-of-alpha-beta-gamma}.

\normalsize
As the lex-largest monomial in $\det D^{(\ell,m)}$
has the same coefficient as in $\det \alpha \cdot \det \beta \cdot \det \gamma$,
the descriptions in \eqref{closer-examination-of-alpha-beta-gamma}
imply that this monomial is a power of $z_\ell$ 
times the product of the lex-largest monomials in 
$$ 
\det D^{(\ell-1,m)}, 
\,\, \det D^{(\ell-1,m-1)}, 
\,\, \det \wedge^m\, {\VDM}^{(\ell-1)}. 
$$
By induction on $\ell$, the coefficient on the lex-largest monomials
in $\det D^{(\ell-1,m)}$ and $\det D^{(\ell-1,m-1)}$
are both $\pm 1$.  For
$\det \wedge^m\,  {\VDM}^{(\ell-1)}$, the Sylvester-Franke Theorem
says $\det \wedge^m A = \left( \det A \right)^{\binom{n-1}{m-1}}$,
and since one has coefficient $\pm 1$ on
the lex-largest monomial 
$z_{\ell-1}^{\ell-2} \cdots z_2^1 z_1^0$ in $\det{\VDM}^{(\ell-1)}$,
the same holds for $\det \wedge^m\, {\VDM}^{(\ell-1)}$.
Thus this also holds for $\det D^{(\ell,m)}$, completing the proof.
\end{proof}

\section*{Acknowledgments}
The authors thank Corrado DeConcini, Paolo Papi, Mark Reeder and John Stembridge for helpful conversations and references, and making them aware of their unpublished work.  They also thank an anonymous referee for helpful comments.


\vspace{-2ex}
\end{document}